\documentclass[10pt]{article}

\usepackage{graphicx}

\usepackage[latin1]{inputenc}
\usepackage[english]{babel}
\usepackage[T1]{fontenc}
\usepackage{lmodern}
\usepackage{graphicx}
\usepackage{amsfonts}
\usepackage{stmaryrd}
\usepackage[nice]{nicefrac}
\usepackage{amsmath,amsthm,amssymb,mathabx}
\usepackage{enumerate}
\usepackage{hyperref}
\usepackage[top=2cm, bottom=2cm, left=2cm , right=2cm]{geometry}

\newtheorem{theorem}{Theorem}
\newtheorem{definition}{Definition}
\newtheorem{lemma}{Lemma}
\newtheorem{proposition}{Proposition}
\newtheorem{corollary}{Corollary}
\usepackage{array}
\usepackage{tikz}

\newcommand{\mc}[1]{\ensuremath{\mathcal{#1}}}
\newcommand{\bs}[1]{\ensuremath{\boldsymbol{#1}}}

\def\ao{$\mathsf{and/or}$ }
\def\a{$\mathsf{and}$ }
\def\o{$\mathsf{or}$ }
\def\true{\mathsf{true}}
\def\false{\mathsf{false}}
\def\rat{\mathsf{rat}}

\newcommand{\I}{\mathcal{I}}
\newcommand{\M}{\mathcal{M}}

\newcommand{\e}{\mathtt{e}}

\newcommand{\IP}{\mathbb{P}}

\newcommand{\G}{\mathtt{(G)}}
\newcommand{\E}{\mathtt{(E)}}
\newcommand{\Cat}{\mathtt{Cat}}
\newcommand{\Lab}{\mathtt{Lab}}
\newcommand{\cst}{\mathtt{cst}}

\vfuzz \hfuzz

\begin{document}

\title{Generalised and Quotient Models for Random And/Or~Trees\\and\\Application to Satisfiability}

\author{Antoine Genitrini\thanks{Sorbonne Universit{\'e}s, UPMC Univ. Paris 06, CNRS, LIP6 UMR 7606, 4 place Jussieu 75005 Paris.
\texttt{Antoine.Genitrini@lip6.fr}.}
~and C\'ecile Mailler\thanks{Department of Mathematical Sciences, University of Bath, BA2 7AY Bath, UK.
\texttt{c.mailler@bath.ac.uk}.}}

\maketitle

\begin{abstract}
This article is motivated by the following satisfiability question:
pick uniformly at random an \ao Boolean expression of length $n$, built on a set of $k_n$ Boolean variables.
What is the probability that this expression is satisfiable? asymptotically when $n$ tends to infinity? %$+\infty$?

The model of random Boolean expressions developed in the present paper is the model of Boolean Catalan trees, 
already extensively studied in the literature for a \emph{constant} sequence $(k_n)_{n\geq 1}$. 
The fundamental breakthrough of this paper is to generalise the previous results for any (reasonable) 
sequence of integers $(k_n)_{n\geq 1}$, which enables us, in particular, to solve the above satisfiability question.

We also analyse the effect of introducing a natural equivalence relation on the set of Boolean expressions.
This new \emph{quotient} model happens to exhibit a very interesting threshold (or saturation) phenomena
at $k_n = \nicefrac{n}{\ln n}$.
 
\vspace{\baselineskip}

\end{abstract}

\section{Introduction}

For several decades, satisfiability problems have been extensively studied by
computer scientists and probabilists, as well as statistical physicists.
In this paper, we focus on the probabilistic version of satisfiability problems:
what is the probability that a \emph{random} Boolean expression is satisfiable?
The answer to this question obviously depends on the distribution 
considered on the set of Boolean expressions.
%In the present paper, we choose such a distribution, inspired by the literature on 
%quantitative logics and use Analytic Combinatorics to study it.

One of the most studied satisfiability problems is the $3$--{\sc sat} problem.
It consists in choosing uniformly at random an expression among conjunctions of $n$ clauses,
each clause being a disjunction of three literals - where literals are chosen among a set of $k_n$ variables and their negations.
What is the probability that such a random Boolean expression is satisfiable? when $n$ tends to infinity?

This question is already partially answered -- see for example~\cite{AM06}:
the following phase transition is proven.
If the ratio $\nicefrac{k_n}{n}$ is small enough, 
then the random expression is satisfiable with probability tending to~$1$ when $n$ tends to infinity, 
whereas if the ratio $\nicefrac{k_n}{n}$ is large enough, then, this probability tends to~$0$.
Refining this statement is the challenging aim of a large literature.

There are many other satisfiability problems.
The $K$--{\sc sat} problem is for example the object of a recent breakthrough 
by Coja-Oghlan and Panagiotou~\cite{CP13} and Coja-Oghlan~\cite{Coja14},
who obtained the existence of a sharp threshold when $K$ tends to infinity.
The $2$-{\sc xorsat} problem is studied by Daud{\'e} and Ravelomanana~\cite{DR11}, using Analytic Combinatorics 
to exhibit and describe precisely a phase transition phenomenon.

\vspace{\baselineskip}
The aim of the present paper is to define and study a new satisfiability model 
(i.e. a new distribution on the set of Boolean expressions) inspired by the 
literature on quantitative logics.

Quantitative logics, which origin might go back to the work of Woods~\cite{Woods},
aims at answering this question:
Which Boolean function does a \emph{random} Boolean expression represent?
Once again, the answer to this question deeply depends
on the model of randomness chosen for Boolean expressions.

The Catalan tree model, first studied by Lefmann and Savick\'y~\cite{LS97}, is defined as follows:
A Boolean tree is a binary plane rooted tree (i.e. a Catalan tree) whose internal nodes are labelled by
the connectives $\mathtt{and}$ or $\mathtt{or}$
and whose leaves are labelled by $k$ variables and their negations.
Pick up uniformly at random a tree among Boolean trees of size $n$, and denote by 
$\mathbb P_{n,k}$ the distribution it induces on the set of Boolean functions.
Lefmann and Savick\'y first proved the existence of a limiting probability distribution $\mathbb P_k$
on Boolean functions when the size $n$ of the random Boolean expression tends to infinity.

Since the seminal paper by Chauvin et al.~\cite{CFGG04},
the Analytic Combinatorics' community aims at understanding better the Catalan tree distribution $\mathbb P_k$ 
(and similarly defined distributions) on the set of Boolean functions.
In particular, Kozik~\cite{kozik08} proves, in the Catalan tree model, 
an asymptotic (when $k$ tends to infinity) relation between the probability of a given function 
and its \emph{complexity}
(i.e. the complexity of a Boolean function being the size of the smallest tree representing it).
His powerful approach, the \emph{pattern theory},
easily classifies and counts large expressions according to specific structural constraints.
It will be generalised in the present paper.

Remark that in the Catalan tree model defined above, 
the size $n$ of the Boolean expressions tend to infinity while the number $k$ of literals labelling them is fixed.
For technical reasons, $k$ is then sent to infinity in order to obtain an asymptotic estimate of the probability of a given Boolean function.
It means that the trees we consider have a lot of repetitions in their leaves: 
it is legitimate to ask if this bias the distribution induced on the set of Boolean functions.
%However the main objection to this model is about the two consecutive limits that cannot be interchanged: 
%in order to obtain quantitative results, a function must be fixed
%and thus we cannot consider functions whose complexity depends on~$n$.
Genitrini and Kozik~\cite{GKZ07,GK12} have proposed another model where 
random Boolean expressions are built on an infinite set of variables. 
This approach avoids the bias induced by letting $n$ tend to infinity while $k$ stays fixed.

%%%%%%%%%%%%%%%%%%%%% Je trouve que ce n'est vraiment pas la peine de parler de papiers sur Shannon ici
%%%%%%%%%%%%%%%%%%%%% Globalement, le problÃ¨me auquel on rÃ©pond n'est pas que l'on se restreint Ã  des fonctions BoolÃ©ennes indÃ©pendantes de k, puisque dans notre travail ici, nos rÃ©sultats ne seront valides que pour des fonctions indÃ©pendantes de n
%%%%%%%%%%%%%%%%%%%%% Ce que l'on fait, c'est juste qu'on Ã©vite plein de rÃ©pÃ©titions artificielles dans les feuilles de arbres et/ou que l'on considÃ¨re.

%According to our knowledge, the two papers that relates the number of variables to the size are~\cite{GG10,GGM14}: 
%they define a large family of functions (i.e. with a strictly positive probability) of small complexity.
%However from these results we cannot derive
%any quantitative results of the probability of small families of functions whose
%complexities depends on $n$.
%Moreover looking at satisfiability problems in this context seems not to be very exciting
%and even to have some signification.

%%%%%%%%%%%%%%%%%%%%% D'autant plus que Shannon a un sens dans ce modÃ¨le puisque c'est ce qu'on va Ã©crire avec Bernhard dans la version longue d'AofA.

Our paper extends the Catalan model in order both
(1) to let $n$ and $k$ tend to infinity together and
(2) to fit in the satisfiability context.

\vspace{\baselineskip}
Following the extended abstract~\cite{GM14}, we also look at the influence of
a natural notion of equivalence on the set of Boolean expressions and functions.
Roughly speaking, we say that two expressions or functions are equivalent 
if the second one can be obtain from the first one by renumbering the variables.
As an example, the expressions $(x_1 \;\mathtt{and}\; x_2)$ and $(x_{12} \;\mathtt{and}\; x_{3})$
are equivalent.

We will describe and study in parallel these two models (with an without equivalence classes) where
the number of variables and the size of expressions jointly tend to infinity.
Since the proofs will be very similar in both models, we will try general notations that fit both models.
The model without equivalence classes will permit, as a corollary to answer 
the satisfiability problem in the context of Catalan Boolean expressions.
It will be very interesting to see that,
although the proofs are completely similar for both models,
the probability distributions induced on the set of Boolean functions behave differently:
the introduction of equivalence classes gives birth to an interesting and quite mysterious threshold phenomenon.

The paper is organised as follows.
In Section~\ref{sec:model} we define our two new models: 
the \emph{generalised} model where the number of variables depends on the size of the considered trees
and the \emph{quotient} model where we introduce a natural equivalence relation on Boolean trees and functions. 
Section~\ref{sec:results} is devoted to stating and discussing our three main results:
the satisfiability question for random Catalan expressions;
the link between the probability of a Boolean function (resp. a class of Boolean functions)  and its \emph{complexity},
both in the generalised and the quotient models.
Section~\ref{sec:technical} and Section~\ref{sec:patterns} contain the technical core of the paper:
Section~\ref{sec:technical} is an analytic part focusing mainly
on the difficulties arising from the introduction of the equivalence relation,
while Section~\ref{sec:patterns} concerns both models and discusses Kozik's pattern theory.
Finally Section~\ref{sec:proba} contains the proofs of our main results.

\section{Description of the two models}\label{sec:model}

\subsection{Contextual definitions}

A {\bf Boolean function} is a mapping from $\{0,1\}^{\mathbb{N}}$ into $\{0,1\}$. 
The two constant functions $(x_i)_{i\geq 1}\mapsto 1$ and $(x_i)_{i\geq 1}\mapsto 0$ are respectively called $\true$ and $\false$. 

An {\bf \ao tree} is a binary plane tree whose leaves are labelled by literals, 
i.e. by elements of $\{x_i, \bar{x}_i\}_{i\in\mathbb{N}}$, 
and whose internal nodes are labelled by the connective \a or the connective $\mathsf{or}$, 
respectively denoted by $\land$ and $\lor$.
We will say that $x_i$ and $\bar{x}_i$ are two different literals 
but they are respectively the positive and the negative version of the same variable $x_i$. 
Every \ao tree is equivalent to a Boolean expression and thus represents a Boolean function: 
for example, the tree in Fig.~\ref{fig:true} is equivalent to the expression 
$([x_1\lor(\lnot{x}_1\lor x_2)]\lor x_3)\lor (x_4\land x_1)$,
where $\lnot x = 1-x$ for all $x\in\{0,1\}$,
and represents the constant function $\true$.

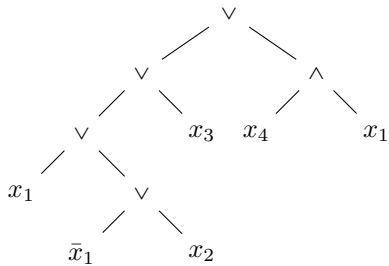
\begin{figure}
\centering
\begin{tikzpicture}[style={level distance=1cm},level 1/.style={sibling distance=2.9cm},level 2/.style={sibling distance=2cm}, scale=0.8]
 \node [circle] (z){$\vee$}
	child {node [circle] (x) {$\vee$}
		child{node [circle] (a) {$\vee$}
				child {node [circle] (y) {$x_1$}}
				child {node [circle] (w) {$\vee$}
					child {node [circle] (b) {$\bar{x}_1$}}
					child {node [circle] (c) {$x_2$}}
					}
				}
		child {node [circle] (d) {$x_3$}}
		}
	child {node [circle] (g) {$\wedge$}
 				child {node [circle] (e) {$x_4$}}
 				child {node [circle] (f) {$x_1$}} 
				}
  ;
\end{tikzpicture}
\caption{An \ao tree computing the constant function $\true$.}
\label{fig:true}
\end{figure}

The {\bf size} of an \ao tree is its number of leaves: remark that,
for all $n\geq 1$, there is infinitely many \ao trees of size $n$.
Finally we define the {\bf tree-structure} of an \ao tree
to be the \ao tree where the labels of the leaves (but not of the internal nodes)
have been removed.

\begin{definition}
The {\bf complexity} of a non constant Boolean function $f$, 
denoted by $L(f)$, is defined to be the size of its {\bf minimal trees}, 
i.e. the size of the smallest trees computing $f$. 
The complexity of $\true$ and $\false$ is defined to be~0.
\end{definition}

Although a Boolean function is defined on an infinite set of variables, 
it may actually depend only on a finite subset of \emph{essential variables}.
\begin{definition} 
Given a Boolean function $f$, we say that the variable $x$ is {\bf essential} for $f$ if, and only if,
$f_{|x\leftarrow0} \not\equiv f_{|x\leftarrow1}$ 
(where $f_{|x\leftarrow \alpha}$ is the restriction of $f$ to the subspace where $x=\alpha$). 
We denote by $E(f)$ the number of essential variables of $f$.
\end{definition}

Remark that the complexity and the number of essential variables of a Boolean function are related by the following
inequalities: $E(f)\leq L(f) \leq 2^{E(f)+2}$ (see e.g. \cite[p. 77--78]{FS09} for the second inequality).
Note that, asymptotically when $E(f)$ tends to infinity a tight asymptotic upper-bound is $\nicefrac{2^{E(f)}}{E(f)}$, 
as proved by Lupanov~\cite{Lupanov58} for the upper bound and Lutz~\cite{Lutz92} for the lower bound.

In the whole paper, our models propose a way to make $n$ and $k$ tend to infinity together:
\begin{definition} 
Let $(k_n)_{n\geq 1}$ be an increasing sequence of integers such that
$k_n$ tends to infinity when $n$ tends to infinity.
\end{definition}

\subsection{The generalised Catalan tree model}

Let us recall the definition of the Catalan tree model defined and studied by 
Paris et al.~\cite{PVW94}, Lefmann \& Savick{\'y}~\cite{LS97}, Chauvin et al.~\cite{CFGG04} and Kozik~\cite{kozik08}.
In those papers, the authors fix an integer $k\geq 1$ and consider the uniform distribution
on \ao trees of size $n$ whose leaf-labels are constrained to be in $\{x_1, \bar x_1, \ldots, x_k, \bar x_k\}$.
They study the induced distribution on the set of Boolean variables and prove that this distribution 
converges to a limit distribution $\mathfrak{p}_k$ when the size $n$ of the trees tends to infinity.
Given a Boolean function $f$, they then prove asymptotic theorems for $\mathfrak{p}_k(f)$ when $k$ tends to infinity. %$k\to+\infty$.
In this approach, the order of the two limits (on $n$ and then on $k$) is a priori important.

We define first the generalised Catalan tree model, that is 
a natural extension of the previous model.

\vspace{\baselineskip}
\noindent The model $\G$ is defined as follows:
\begin{enumerate}[(1)]
\item consider the uniform distribution on \ao trees of size $n$ which
leaf-labels belong to $\{x_1, \bar x_1$, $\ldots$, $x_{k_n}, \bar x_{k_n}\}$,
\item denote by $\mathbb P_n$ the distribution it induces on the set of Boolean functions,
and call this new distribution the {\bf generalised Catalan tree distribution}.
\end{enumerate}
Remark that there are $A_n$ \ao trees of size $n$ labelled with $k_n$ variables, with 
\begin{equation}\label{eq:cat}
A_n= 2^{n-1} (2k_n)^n \cdot\Cat_n, \hspace*{2cm} \text{where } \Cat_n = \frac1{n}\binom{2n-2}{n-1},
\end{equation}
i.e. $\Cat_n$ is the number of binary plane trees having $n$ leaves.

For all Boolean function $f$, we denote by $A_n(f)$ the number of \ao trees of size $n$
labelled with $k_n$ variables that compute $f$.
Thus, by definition,
\[\mathbb P_n (f) = \frac{A_n(f)}{A_n}.\]

\subsection{The quotient Catalan tree model}

A second natural generalisation of the Catalan tree model is obtained
by introducing equivalence classes of Boolean trees and functions.
The idea is the following: the functions 
$(x_i)_{i\geq 1} \mapsto x_1 \land x_2$ and $(x_i)_{i\geq 1} \mapsto x_{38} \land \bar x_{12}$
can be seen as two realisations of the function \emph{conjunction}.

Informally, two \ao trees are equivalent if the leaves of the first one can be relabelled (and negated) without collision
in order to obtain the second tree.
We define formally this equivalence relation as follows.
\begin{definition}
Let $A$ and $B$ be two \ao trees.
Trees $A$ and $B$ are {\bf equivalent} if 
\begin{enumerate}[(i)]
\item their tree-structures are identical;
\item two leaves are labelled by the same variable in $A$ if and only of they are labelled by the same variable in $B$;
\item two leaves are labelled by the same literal in $A$ if and only of they are labelled by the same literal in $B$.
\end{enumerate}
\end{definition}
This equivalence relation on Boolean trees \emph{induces straightforwardly
an equivalence relation on Boolean functions}.
Note that all functions of an equivalence class have the same complexity and the same number of essential variables.
In the following, we will denote by $\langle f\rangle$ the equivalence class of the function~$f$.
We denote by $L\langle f \rangle= L(f)$
(resp. $E\langle f \rangle=E(f)$)  the common complexity (resp. number of essential variables)
of the elements of $\langle f\rangle$. 

\begin{definition}
Let $\langle f\rangle$ be a class of Boolean functions.
The {\bf multiplicity} of the class $\langle f \rangle$, is given by
\[R\langle f\rangle = L\langle f \rangle - E\langle f \rangle.\]
It corresponds to the number of repetitions of variables in a minimal 
tree of a function from $\langle f\rangle$.
\end{definition}

Recall that $(k_n)_{n\geq 1}$ is an increasing sequence of integers that tends to infinity 
when $n$ tends to infinity. 
In the following, we only consider equivalence classes of trees having at least one element 
whose leaf-labels are in $\{x_1, \bar x_1$, $\ldots$, $x_{k_n}, \bar x_{k_n}\}$.
It means that we restrict ourselves to trees of size $n$ labelled by at most $k_n$ different variables.
Note that if $k_n \geq n$ for all $n\geq 1$, 
this is not a restriction because a tree of size $n$ cannot contain more that $n$ 
different leaf-labels.

\vspace{\baselineskip}
The model $\E$ is defined as follows:
\begin{enumerate}[(1)]
\item consider the uniform distribution on classes of equivalence of trees of size $n$ 
(labelled with at most $k_n$ different variables),
\item the distribution it induces on the set of equivalence classes of Boolean function is denoted by 
$\mathbb P_n$ and called the {\bf quotient Catalan tree distribution}.
\end{enumerate}

We denote by $A_n$ the number of equivalence classes of trees of size $n$ 
(in which at most $k_n$ different variables appear as leaf-labels).
Given a class of Boolean functions $\langle f\rangle$, 
we denote by $A_n\langle f\rangle$ the number of equivalence classes
of trees of size~$n$ (labelled with at most $k_n$ different variables) 
that compute a function of $\langle f\rangle$.
We thus have
\[\mathbb P_n\langle f\rangle = \frac{A_n\langle f\rangle}{A_n}.\]

\begin{proposition}\label{prop:enum}
The number of classes of trees of size $n$ satisfies:
\[A_n = \Cat_n \cdot \sum_{p=1}^{k_n} {n \brace p} 2^{2n-1-p},\]
where $\Cat_n$ is the number of (unlabelled) binary planar trees having $n$ leaves (cf.~Equation~\eqref{eq:cat}),
and where ${n \brace p}$ is the Stirling number of the second kind.\footnote{In Proposition~\ref{prop:enum}, ${n \brace p}$ is the number 
of partitions of $n$ objects in $p$ non-empty subsets (see e.g. \cite[p. 735--737]{FS09}).}
\end{proposition}

\begin{proof}
An equivalence class of \ao trees can be seen as
\begin{itemize}
\item a binary plane tree (factor $\Cat_n$)
\item whose internal nodes are labelled by \a and \o connectives (factor $2^{n-1}$),
\item whose leaves are partitioned onto $1\leq p \leq k_n$ parts (factor ${n \brace p}$),
\item each of these parts being then partitioned onto two parts (one on them being possibly empty: factor $2^{n-p}$).
\end{itemize}
\end{proof}

{\bf Remark on notations:} We have already used the notation $A_n$ to define the model $\G$. 
We will keep the same notation for these two distinct objects because they will have the same role in the proofs.
But formally, we have
\[A_n^{\G} = \Cat_n \cdot 2^{2n-1} \cdot k_n^n \quad \text{ and }\quad A_n^{\E} = \Cat_n \cdot \sum_{p=1}^{k_n} {n \brace p} 2^{2n-1-p}.\]

\section{Main results and discussion}\label{sec:results}

We have defined the two models we are interested in: 
the generalised and the quotient Catalan trees distributions.
Both distributions are called $\mathbb P_n$ for simplicity's sake, 
but we will use $\mathbb P_n^{\G}$ and $\mathbb P_n^{\E}$ when the precision is needed.
The aim of this paper is to study the behaviour of both distributions when the size $n$ 
of the considered trees tends to infinity.

Let us remark that the distribution induced by $\G$ is based on an uniform distribution among trees
of the same size. But the distribution induced by $\E$ lies on an uniform distribution among classes
of trees of the same size. Obviously both induced distributions on Boolean functions are distinct.

\begin{theorem}[Model $\G$]\label{thm:thetaG}
Let $(k_n)_{n\geq 1}$ be an increasing sequence of integers tending to infinity when $n$ tends to infinity.
For all Boolean functions $f$, there exists a positive constant $\alpha^\G_f$ such that, 
asymptotically when $n$ tends to infinity, %$n\to+\infty$,
\[\mathbb P_n(f) \sim \alpha^\G_f \cdot\left(\frac1{k_n}\right)^{L(f)+1}.\]
\end{theorem}

This result has an interesting corollary concerning the Catalan-{\sc sat} problem:
recall that a Boolean expression is said \emph{satisfiable}
if it does not represent the constant function $\false$.
\begin{corollary}[Catalan-{\sc sat}]\label{thm:satis}
Let $(k_n)_{n\geq 1}$ be an increasing sequence of integers tending to infinity when $n$ tends to infinity.
Pick up uniformly at random an \ao tree of size $n$ with leaf-labels in $\{x_1, \bar x_1$, $\ldots$, $x_{k_n}, \bar x_{k_n}\}$.
This random \ao tree is equivalent to a Boolean expression that is satisfiable with probability tending to~$1$
when $n$ tends to infinity.
\end{corollary}

\begin{theorem}[Model $\E$]\label{thm:thetaE}
Let $(k_n)_{n\geq 1}$ be an increasing sequence of integers tending to infinity when $n$ tends to infinity.
There exists a sequence $(M_n)_{n\geq 1}$ such that $M_n\sim_{n\rightarrow \infty} \frac{n}{\ln n}$% (when $n$ tends to infinity)
and such that,
for all fixed equivalence classes of Boolean functions $\langle f\rangle $,
there exists a positive constant $\alpha^\E_{\langle f\rangle}$ satisfying:
\begin{enumerate}[(i)]
\item if, for all sufficiently large $n$, $k_n \leq M_n$, then, asymptotically when $n$ tends to infinity,
\[\mathbb{P}_n \langle f \rangle \sim \alpha^\E_{\langle f\rangle}\cdot  \left( \frac{1}{k_{n+1}} \right)^{R\langle f \rangle+1} ;\]
\item if, for all sufficiently large $n$, $k_n \geq M_n$, then, asymptotically when $n$ tends to infinity,
\[\mathbb{P}_n \langle f \rangle \sim \alpha^\E_{\langle f\rangle} 	\cdot \left(\frac{\ln n}{n}\right)^{R\langle f \rangle+1}.\]
\end{enumerate}
\end{theorem}
First note, that we could give some corollary about satisfiability for the second model $\E$ too.
However, in the classical context of SAT problems, there are no quotient formulas. So we 
omit this by-product.

Let us discuss these results in view of the classical Catalan tree distribution
studied by~\cite{CFGG04} and~\cite{kozik08}:
let us recall briefly its definition.
Let $k\geq 1$ be an integer. We denote by $T_{n,k}$ the number of trees of size $n$, 
with leaf-labels in $\{x_1, \bar x_1, \ldots, x_k, \bar x_k\}$.
Given a Boolean function $f$, we denote by $T_{n,k}(f)$ the number of such trees computing $f$.
The Catalan distribution is thus defined by, for all Boolean functions $f$,
\[\mathfrak{p}_k(f) := \lim_{n\to+\infty} \frac{T_{n,k}(f)}{T_{n,k}}.\]
The existence of the above limit is proved in~\cite{LS97} or~\cite{CFGG04}.
Kozik proved:
\begin{theorem}[Kozik~\cite{kozik08}]\label{thm:kozik}
Let $k$ be a fixed positive integer.
For all Boolean functions $f$, there exists a positive constant $c_f$ such that
\[\mathfrak p_k(f) \sim_{k\rightarrow \infty} c_f \cdot \left(\frac1{k}\right)^{L(f)+1}.\] 
\end{theorem}
As one can see Theorems~\ref{thm:thetaG} and~\ref{thm:kozik} are very similar, 
and we will see that their proofs are also very similar after having observed a simple but fundamental trick:
one has to consider separately the tree-structure of an \ao tree and its leaf-labelling. 
It was not clear before this work how to generalise Kozik's proof 
in order to tackle the Catalan-{\sc sat} problem (cf. Corollary~\ref{thm:satis}).

\vspace{\baselineskip}
Introducing equivalence classes makes things different, 
and an interesting threshold effect appears (see Theorem~\ref{thm:thetaE}).
We still have no intuition for this threshold. Obviously we will see in the proof where it comes from.

In the classical Catalan tree model, each Boolean function is studied separately 
instead of being considered among its equivalence class. 
We can translate the result obtained by Kozik in terms of equivalence classes 
by summing over all Boolean functions belonging to a given equivalence class: 
note that there are $\binom{k}{E(f)}2^{E(f)}$ functions in the equivalence class of $f$. 
Therefore, the result of Kozik is equivalent to:
for all classes $\langle f\rangle$, there exists a constant $c_{\langle f \rangle}$ such that, 
asymptotically when $k$ tends to infinity,
\[\lim_{n\to+\infty} \mathfrak p_{n,k}\langle f\rangle 
\sim c_{\langle f\rangle} \left(\frac1{k}\right)^{L(f) - E(f)+1}
= c_{\langle f\rangle} \left(\frac1{k}\right)^{R\langle f\rangle+1}.\]
The classical Catalan tree distribution can be seen as a degenerate case of our model where there
exists a fixed integer $k$ such that $k_n=k$ for all $n\geq 1$.
Recall that we assume in the present paper that $k_n$ tends to infinity when $n$ tend to infinity: 
the case $k_n=k$ is thus not a particular case of our results, but only a degenerate one.

Once again, the proof of Theorem~\ref{thm:thetaE} relies on similar ideas as Kozik's proof of Theorem~\ref{thm:kozik}.
To emphasise the similarities between the proof of our two main theorems (Theorems~\ref{thm:thetaG} and~\ref{thm:thetaE}),
we will develop their proofs together in Section~\ref{sec:proba}.

%Concerning the infinite context~\cite{GKZ07,GK12} $k_n=+\infty$, we already noticed that the cases such that $k_n$ is larger than
%$n$ are equivalent to the model $k_n=n$, even if $k_n=+\infty$. 
%Therefore, this infinite context is actually the extreme case $k_n=n$ of our model,
%and is thus fully treated in the present paper.
%In this specific setting, the Stirling numbers introduced in Proposition~\ref{prop:enum} induce Bell numbers, that naturally 
%appear in~\cite{GKZ07,GK12}.

\section{Technical key point}\label{sec:technical}

As we already mentioned, the key idea of this paper is to separate the tree-structure
of an \ao tree and its leaf-labelling. 
Recall that
\[A_n^{\G} = 2^{n-1}\Cat_n \cdot (2k_n)^n  \quad \text{ and }\quad A_n^{\E} = 2^{n-1}\Cat_n \cdot \sum_{p=1}^{k_n} {n \brace p} 2^{n-p}.\]
For all $m,n\geq 1$, let us denote by
\[
\Lab_{n,m} :=\left\{
\begin{array}{lll}
(2m)^n & \text{ in model }\G ;\\
& \\
\displaystyle 2^n\cdot \sum_{p=1}^m {n\brace p} 2^{-p} & \text{ in model }\E.
\end{array}\right.\]
In both models, $\Lab_{n,m}$ corresponds to the number of ways to label the $n$ leaves with $m$ variables, thus
\[A_n = 2^{n-1}\Cat_n \cdot \Lab_{n,k_n}.\]
Finally, let us introduce the key quantity
\[\rat_n:= \frac{\Lab_{n-1,k_{n}}}{\Lab_{n,k_{n}}}.\]
Note that in the model $\G$, the quantity $\nicefrac{1}{\rat_n}=2k_{n}$ corresponds to the number of the possible labellings of the $(n+1)^\text{th}$
leaf once the other leaves are already labelled.
In the model $\E$, the leaf-labellings are not longer independent and this quantity $\nicefrac{1}{\rat_n}$ is thus less explicit.
A detailed analysis of this quantity is needed in the following.
This section is devoted to its asymptotic analysis.
\begin{proposition}\label{prop:labelling2}
Let $(k_n)_{n\geq 1}$ be an increasing sequence of integer tending to infinity when $n$ tends to infinity.% $n\to+\infty$. 
\begin{itemize}
\item[$\G$] For all integer $p$,
\[\frac{\Lab_{n-p,k_n}}{\Lab_{n,k_n}} = \frac1{(2k_n)^p}.\]
\item[$\E$] There exists a sequence $(M_n)_{n\geq 1}$ with $M_n\sim_{n\rightarrow\infty} \frac{n}{\ln n}$  and such that,
for all integer $p$,
asymptotically when $n$ tends to infinity,
\[\frac{\Lab_{n-p,k_n}}{\Lab_{n,k_n}} = 
\left\{\begin{array}{lll}
\frac{1+o(1)}{(2k_n)^p} & \hspace*{0.5cm} & \text{ if } k_n\leq M_n \text{ for large enough }n;\\
&&\\
(1+o(1))\ \left(\frac{\ln n}{2n}\right)^p && \text{ if } k_n\geq M_n \text{ for large enough }n.
\end{array}\right.\]
\end{itemize}
\end{proposition}
In particular, taking $p=1$ gives
\begin{proposition}\label{prop:labelling}
Let $(k_n)_{n\geq 1}$ be an increasing sequence of integer tending to infinity when $n$ tends to infinity.% $n\to+\infty$. 
\begin{itemize}
\item[$\G$]
$\displaystyle \rat_n = \frac1{2k_n}$.
\item[$\E$] There exists a sequence $(M_n)_{n\geq 1}$ with $M_n\sim_{n\rightarrow\infty} \frac{n}{\ln n}$  and such that,
asymptotically when $n$ tends to infinity,
\[\rat_n = 
\left\{\begin{array}{lll}
\frac{1+o(1)}{2k_n} & \hspace*{0.5cm} & \text{ if } k_n\leq M_n \text{ for large enough }n;\\
&&\\
(1+o(1))\ \frac{\ln n}{2n} && \text{ if } k_n\geq M_n \text{ for large enough }n.
\end{array}\right.\]
\end{itemize}
\end{proposition}

Remark that, with this definition of $\rat_n$, Theorems~\ref{thm:thetaG} and~\ref{thm:thetaE} can be rephrased as:
for all Boolean functions $f$, there exists constants 
\[\mathbb P_n^{\G}(f) \sim \lambda_f \cdot \rat_n^{L(f)+1},\]
and
\[\mathbb P_n^{\E}\langle f\rangle \sim \lambda_{\langle f \rangle} \cdot \rat_n^{R\langle f\rangle+1}.\]

The proof of Proposition~\ref{prop:labelling} $\G$ is obvious and 
the rest of this section is devoted to the more technical proof of Proposition~\ref{prop:labelling} $\E$.

%\subsection{Threshold induced by $k_n$'s behaviour}

%\begin{definition}\label{df:B_n}
%Let us define the following quantity:for all integers $m$ and $n$, 
%\[B_{n,k_n} = \sum_{p=1}^{k_n} {n \brace p} 2^{-p}.\]
%\end{definition}
%The number $B_{n,k_n}$ quantitatively represents the labelling constraints of leaf-labelling by variables (cf. Proposition~\ref{prop:enum}).

The following proposition, which can be seen as some particular case of Bonferroni inequalities allows to exhibit bounds on $\Lab_{n,k_n}$.
\begin{proposition}[cf. for example~\cite{Sibuya88}]\label{fact:inegalite}
For all $n\geq 1$, for all $p\in\{1,\ldots,n\}$,
\[\frac{p^n}{p!} - \frac{(p-1)^n}{(p-1)!} \leq \left\{\begin{matrix}n\\p\end{matrix}\right\} \leq \frac{p^n}{p!}.\]
\end{proposition}

In view of these inequalities and of the expression of $\Lab_{n,k_n}$,
both the following sequences naturally appear:
\begin{lemma}\label{lem:croissance}
Let $n$ be a positive integer.
\begin{enumerate}[(i)]
\item The following sequence is unimodal:
\[\left(a_p^{(n)}\right)_{p\in\{1,\ldots,n\}} = \left(\frac{p^n}{p!}2^{-p}\right)_{p\in\{1,\ldots,n\}},\]
i.e. there exists an integer $M_n$ such that $\left(a_p^{(n)}\right)_p$ is strictly increasing
on $\{1, \dots,  M_n\}$ and strictly decreasing on $\{M_n +1,\dots, n\}$.
\item Moreover, the sequence $(M_n)_n$ is increasing and asymptotically satisfies:
\[M_n \sim_{n\rightarrow\infty} \frac{n}{\ln n}.\]
\end{enumerate}
\end{lemma}
\begin{proof}
{\it (i)} 
Let us prove that the sequence $\left(a_p^{(n)}\right)_{1\leq p\leq n}$ is log-concave, 
i.e. that the sequence $\left(\frac{a_{p+1}^{(n)}}{a_p^{(n)}}\right)_{1\leq p\leq n-1}$ is decreasing.  
Let $p$ be an integer in $\{1,\dots, n-1\}$. By Definition of $a_p^{(n)}$:
\[\frac{a_{p+1}^{(n)}}{a_p^{(n)}} = \left(\frac{p+1}{p}\right)^n\cdot \frac1{2(p+1)},\]
and consequently, for all $n\geq 0$,
\[\frac{a_{p+1}^{(n)}}{a_p^{(n)}} > 1 \iff%\Longleftightarrow
 n\ln\left(\frac{p+1}{p}\right)-\ln(2p+2) > 0.\]
The function $\phi_n \,:\, p\mapsto n\ln\left(\frac{p+1}{p}\right)-\ln(2p+2)$ is strictly decreasing.
Note that both $\phi_n(1)$ and $\phi_n(n-1)$ are tending to infinity when $n$ tends to infinity.
Then, for all $n$ large enough,
there exists a unique $M_n$ such that $\left(a_p^{(n)}\right)_p$ is strictly increasing on $\{1,\dots,M_n\}$
and strictly decreasing on $\{M_n+1,\dots, n\}$. Let us suppose $n$ large enough for the rest of the proof.\\

{\it (ii)} Let us denote by $x_n$ the single solution of equation:
\begin{equation}\label{eq:max}
\left(\frac{x+1}{x}\right)^n\cdot \frac1{2(x+1)} = 1, \hspace*{1cm}\text{when it exists.}
\end{equation}
First remark that the sequence $(x_n)_{n\geq 1}$ is increasing.
We indeed know: $\phi_n(x_n) = 0$ and $\phi_{n+1}(x_{n+1}) = 0$,
which implies that $\phi_n(x_{n+1}) = -\ln \left(1+\frac{1}{x_{n+1}}\right) < 0$.
Therefore, since for each~$n$, the function $\phi_n$ is decreasing,
we have that $x_{n+1}\geq x_n$, for all large enough $n$.
Therefore, the sequence $(M_n)_{n\geq 1}$ is asymptotically increasing.

Since, asymptotically when $n$ tends to infinity,
\[\left(\frac{\frac{n}{\ln n}+1}{\frac{n}{\ln n}}\right)^n\cdot \frac1{2(\frac{n}{\ln n}+1)} \sim \frac{\ln n}{2},\]
we have that $n / \ln n\leq x_n$ and therefore,
$x_n$ tends to infinity. Thus,
Equation~(\ref{eq:max}) evaluated in $x_n$ is equivalent to
\begin{equation}\label{eq:M_n}
n\ln\left(1+\frac1{x_n}\right) = \ln 2 +\ln (x_n+1),
\end{equation}
which implies $x_n\ln x_n \sim n$, when $n$ tends to infinity.
We easily deduce from this asymptotic relation that
$\ln x_n \sim \ln n$ and that $x_n \sim \frac{n}{\ln n}$ when $n$ tends to infinity.
Since $M_n = \lfloor x_n \rfloor$, we conclude that $M_n \sim \nicefrac{n}{\ln n}$, when $n$ tends to infinity.
\end{proof}

We are now ready to understand the asymptotic behaviour of $\Lab_{n,k_n} / 2^n$:
\emph{roughly speaking, asymptotically, the sum $\Lab_{n,k_n} / 2^n$
does essentially only depend on the terms around $M_n$.}

\begin{lemma}\label{lem:u_n}
Let $(u_n)_{n\geq 1}$ be an increasing sequence such that $u_n\leq n$ for all integer $n\geq 1$ and $u_n$
tends to infinity when $n$ tends to infinity.
\begin{enumerate}[(i)]
\item If, for all large enough $n$, $u_n\leq M_n$, then, for all sequences $(\delta_n)_{n\geq 1}$
such that $\delta_n = o(u_n)$ and $\frac{u_n\sqrt{\ln u_n}}{\sqrt{n}} = o(\delta_n)$, we have, asymptotically when $n$ tends to infinity,
\begin{equation}\label{u_n_petit}
\frac{\Lab_{n,u_n}}{2^n} = (1+o(1))\ \sum_{p=u_n-\delta_n}^{u_n}\frac{p^n}{p!}2^{-p}.
\end{equation}
\item If, for large enough $n$, $u_n \geq M_n$, then, for all sequences $(\delta_n)_{n\geq 1}$ such that $\delta_n = o(u_n)$
 and $\frac{u_n\sqrt{\ln u_n}}{\sqrt{n}} = o(\delta_n)$, for all sequences $(\eta_n)_{n\geq 1}$
 such that $\eta_n = o(M_n)$,
 $\lim_{n\to+\infty}\frac{\eta_n^2}{M_n} = +\infty$
 and $\sqrt{M_n\ln (u_n-M_n)}=o(\eta_n)$,
 we have, asymptotically when $n$ tends to $+\infty$,
\begin{equation}\label{u_n_grand}
\frac{\Lab_{n,u_n}}{2^n}  =(1+o(1))\ \sum_{p=M_n-\delta_n}^{\min\{M_n + \eta_n, u_n\}} \frac{p^n}{p!}2^{-p}.
\end{equation}
\end{enumerate}
\end{lemma}

\begin{proof}[Proof of Lemma~\ref{lem:u_n} {\it (i)}]
Via Proposition~\ref{fact:inegalite}, we can bound $\frac{\Lab_{n,u_n}}{2^n} $: for all $n\geq 1$,
\begin{equation}\label{eq:infini_bounds}
\frac12 \cdot \sum_{p=1}^{u_n-1} \frac{p^n}{p!\ 2^p} + \frac{u_n^n}{u_n!\ 2^{u_n}} \leq \frac{\Lab_{n,u_n}}{2^n} 
\leq \sum_{p=1}^{u_n} \frac{p^n}{p!\ 2^p}.
\end{equation}

Let us assume that $u_n\leq M_n$ for all large enough $n$, and let us prove that the two bounds
of Equations~\eqref{eq:infini_bounds} are of the same asymptotic order when $n$ tends to infinity.

Denote, for all integer $m\geq 1$, $S_m=\sum_{p=1}^{m} a_p^{(n)}$.
Thus Equations~\eqref{eq:infini_bounds} implies
\[\frac{S_{u_n}}{2} \leq \frac{\Lab_{n,u_n}}{2^n}  \leq S_{u_n} .\]
Let us split the sum $S_{u_n}$ into two parts: the last $\delta_n$ summands, and the rest.
\[S_{u_n} = S_{u_n-\delta_n-1} + \sum_{p=u_n-\delta_n}^{u_n} a_{p}^{(n)}.\]
By assumption, $\delta_n = o(u_n)$ and we therefore can choose $n$ large enough such that $u_n>\delta_n$.
Let us prove that $S_{u_n-\delta_n-1}$ is negligible in front of $a_{u_n}$,
and thus in front of $\sum_{p=u_n-\delta_n}^{u_n} a_p^{(n)}$.
Recall that $\left(a_p^{(n)}\right)_{p\geq 1}$ is increasing on $\{1,\ldots, M_n\}$, which implies
\[S_{u_n-\delta_n-1} \leq u_n \cdot a_{u_n-\delta_n}.\]
For all large enough $n$, via Stirling formula, we deduce:
\begin{align*}
\frac{a_{u_n-\delta_n}}{a_{u_n}}
& = 2^{\delta_n}\left(\frac{u_n-\delta_n}{u_n}\right)^n \frac{u_n !}{(u_n-\delta_n)!}
= \left(\frac{2u_n}{\e}\right)^{\delta_n}\left(\frac{u_n-\delta_n}{u_n}\right)^{n-u_n+\delta_n-\frac12}(1+o(1))\\
& = \exp\left[\delta_n\ln \left(\frac{2u_n}{\e}\right) + \left(n-u_n+\delta_n-\frac12\right)\ln\left(1-\frac{\delta_n}{u_n}\right)+o(1)\right].
\end{align*}
Since $\delta_n = o(u_n)$, we get
$\ln\left(1-\frac{\delta_n}{u_n}\right) = -\frac{\delta_n}{u_n} - \frac{\delta_n^2}{2u_n^2} + o\left(\frac{\delta_n^2}{u_n^2}\right)$.
Moreover, $u_n \leq M_n$ thus,
\[
\frac{a_{u_n-\delta_n}}{a_{u_n}}
 = \exp\left[\delta_n\ln 2 +\delta_n \ln u_n - \frac{n\delta_n}{u_n} - \frac{n\delta_n^2}{2u_n^2}+ o\left(\frac{n\delta_n^2}{u_n^2}\right)\right].
\]
Therefore, by using $u_n\leq M_n$, and Equation~\eqref{eq:M_n}, we deduce $\frac{n}{M_n} \geq \ln 2 + \ln M_n$,
\begin{align*}
\frac{a_{u_n-\delta_n}}{a_{u_n}}
&  \leq \exp\left[\delta_n\ln 2 +\delta_n \ln M_n - \frac{n\delta_n}{M_n} - \frac{n\delta_n^2}{2u_n^2}+ o\left(\frac{n\delta_n^2}{u_n^2}\right)\right]\\
& \leq \exp\left[- \frac{n\delta_n^2}{2u_n^2}+ o\left(\frac{n\delta_n^2}{u_n^2}\right)\right].
\end{align*}
From the assumption $\frac{u_n\sqrt{\ln u_n}}{\sqrt{n}} = o(\delta_n)$, we deduce $\ln u_n = o\left(\frac{n\delta_n^2}{u_n^2}\right)$, thus we can conclude
\[\frac{S_{u_n-\delta_n-1}}{a_{u_n}} \leq u_n \frac{a_{u_n-\delta_n}}{a_{u_n}}
 \leq \exp\left[\ln u_n - \frac{n\delta_n^2}{2u_n^2}+ o\left(\frac{n\delta_n^2}{u_n^2}\right)\right] = o(1).\]
And consequently, we get $S_{u_n} \sim_{n\rightarrow \infty} \sum_{p=u_n-\delta_n}^{u_n} a_{p}^{(n)}$.
\end{proof}

\begin{proof}[Proof of Lemma~\ref{lem:u_n},  {\it (ii)}]
Assume that $u_n\geq M_n$ for all large enough $n$.
Let us split the sums of the lower and upper bounds of Equations~\eqref{eq:infini_bounds} into three parts:
the first from index 1 to $M_n-\delta_n-1$, the second from index $M_n-\delta_n$ to $M_n + \eta_n$,
and the third from index $M_n + \eta_n+1$ to $u_n$.
Remark that, if $u_n\leq M_n+\eta_n$, then the third part is empty and the second one is truncated:
\[S_{u_n} = S_{M_n-\delta_n-1} + \sum_{p=M_n-\delta_n}^{M_n+\eta_n} a_p^{(n)} + \sum_{p=M_n+\eta_n+1}^{u_n} a_p^{(n)}.\]
By arguments similar to those developed in the proof of assertion {\it (i)},
we can prove that $S_{M_n-\delta_n-1}$ is negligible in front of $a_{M_n}^{(n)}$,
and thus in front of $\sum_{p=M_n-\delta_n}^{M_n+\eta_n} a_p^{(n)}$.
Therefore, if $u_n\leq M_n+\eta_n$, assertion {\it (ii)} is proved.
Let us now assume that $u_n\geq M_n+\eta_n+1$: to end the proof,
we prove that $\sum_{p=M_N+\eta_n+1}^{u_n} a_p^{(n)}$ is negligible in front of $a_{M_n}^{(n)}$,
and thus in front of $\sum_{p=M_n-\delta_n}^{M_n+\eta_n} a_p^{(n)}$.\\

In view of Lemma~\ref{lem:croissance}, we have
\[\sum_{p=M_n+\eta_n+1}^{u_n} a_p^{(n)} \leq (u_n-M_n-\eta_n) \cdot a_{M_n+\eta_n}^{(n)}.\]
Via Stirling formula,
\begin{align*}
\frac{a_{M_n+\eta_n}^{(n)}}{a_{M_n}^{(n)}}
& = 2^{-\eta_n} \left(\frac{M_n+\eta_n}{M_n}\right)^n \frac{M_n!}{(M_n+\eta_n)!}
= \left(\frac{2(M_n+\eta_n)}{\e}\right)^{-\eta_n}   \left(\frac{M_n+\eta_n}{M_n}\right)^{n-M_n-\frac12} (1+o(1))\\
& = \exp\left[-\eta_n\ln \left(\frac{2(M_n+\eta_n)}{\e}\right)
+ \left(n-M_n-\frac12\right)\ln \left(1+\frac{\eta_n}{M_n}\right)+o(1)\right].
\end{align*}
Since $\ln \left(1+\frac{\eta_n}{M_n}\right) \leq \frac{\eta_n}{M_n}$, we get:
\begin{align*}
\frac{a_{M_n+\eta_n}^{(n)}}{a_{M_n}^{(n)}} 
& \leq \exp\left[-\eta_n\ln 2 + \eta_n - \eta_n\ln (M_n+\eta_n) + \frac{\eta_n}{M_n}(n-M_n-\frac12)+o(1)\right]\\
& = \exp\left[-\eta_n\ln 2  - \eta_n\ln (M_n+\eta_n) + \frac{n\eta_n}{M_n} + o(1)\right].
\end{align*}
Our assumption states $\frac{\eta_n}{M_n} = o(1)$, thus
\begin{align*}
\frac{a_{M_n+\eta_n}^{(n)}}{a_{M_n}^{(n)}} 
& \leq \exp\left[-\eta_n\ln 2 - \eta_n\ln M_n - \eta_n\ln\left(1+\frac{\eta_n}{M_n}\right) + \frac{n\eta_n}{M_n} + o(1)\right]\\
& = \exp\left[-\eta_n\ln 2 - \eta_n\ln M_n - \frac{\eta^2_n}{M_n} + \frac{n\eta_n}{M_n} + \mathcal{O}\left(\frac{\eta_n^3}{M^2_n}\right)\right]
\end{align*}
Since $M_n = \lfloor x_n \rfloor$, we have
\[n\ln\left(1+\frac1{x_n}\right) = n\left(\frac1{M_n} - \frac{1}{2M_n^2} + \mathcal{O}\left(\frac{1}{M_n^3}\right)\right),\]
therefore
\[\ln 2 + \ln (x_n+1) = \ln 2 + \ln M_n + \mathcal{O}\left(\frac1{M_n}\right).\]
Equation~\eqref{eq:M_n} implies:
\begin{align*}
\frac{n}{M_n} &= \ln 2 + \ln M_n + \frac{n}{2M_n^2} + \mathcal{O}\left(\frac{n}{M_n^3}\right) + \mathcal{O}\left(\frac1{M_n}\right) \\
& =  \ln 2 + \ln M_n + \frac{n}{2M_n^2} + \mathcal{O}\left(\frac{n}{M_n^3}\right),
\end{align*}
because $\frac1{M_n} = o(\frac{n}{M_n^3})$. Thus, we conclude
\begin{align*}
\frac{a_{M_n+\eta_n}^{(n)}}{a_{M_n}^{(n)}} 
& \leq \exp\left[ - \frac{\eta_n^2}{M_n} + \mathcal{O}\left(\frac{\eta_n^3}{M_n^2}\right) + \mathcal{O}\left(\frac{n\eta_n}{M^3_n}\right)\right]
= \exp\left[- \frac{\eta_n^2}{M_n} + o\left(\frac{\eta_n^2}{M_n}\right)\right],
\end{align*}
because, from assumption: $\sqrt{M_n\ln (u_n-M_n)}=o(\eta_n)$, we deuce $\sqrt{M_n}=o(\eta_n)$.
Finally we get
\begin{align*}
\frac{\sum_{p=M_n+\eta_n+1}^{u_n} a_p^{(n)}}{a_{M_n}^{(n)}} 
& \leq (u_n-M_n-\eta_n) \frac{a_{M_n+\eta_n}^{(n)}}{a_{M_n}^{(n)}}
\leq \exp\left[\ln (u_n-M_n) - \frac{\eta_n^2}{M_n} + o\left(\frac{\eta_n^2}{M_n}\right)\right] = o(1),
\end{align*}
since, by assumption, $\sqrt{M_n\ln (u_n-M_n)}=o(\eta_n)$.
Therefore, asymptotically when $n$ tends to infinity,
\[S_{u_n}\sim \sum_{p=M_n-\delta_n}^{M_n+\eta_n} a_p^{(n)},\]
which concludes the proof.
\end{proof}

We are now ready for the proof of Proposition~\ref{prop:labelling2}: let us decompose this proof in the
two following Lemmas~\ref{lem:kpetit} and~\ref{lem:kgrand}:
\begin{lemma}\label{lem:kpetit}
Let $(k_n)_{n\geq 1}$ be a sequence of integerssuch that $k_n \leq M_n$ for large enough $n$, 
then, for all integer $p$, asymptotically when $n$ tends to infinity,
\[\frac{\Lab_{n-p, k_n}}{\Lab_{n,k_n}} =(1+o(1))\ \left(\frac1{(2k_n)^p}\right).\]
\end{lemma}

\begin{proof}
{\bf (i) Let us first assume that $\bs{k_n\leq M_{n-p}}$.}
Let $(\delta_n)_{n\geq 1}$ an integer-valued sequence such that $\delta_n = o(k_n)$
and $\frac{k_n\sqrt{\ln k_n}}{\sqrt{n}} = o(\delta_n)$ when $n$ tends to infinity.
Lemma~\ref{lem:u_n} applied to $u_n = k_n$ gives, asymptotically when $n$ tends to infinity,
\[\frac{\Lab_{n,k_n}}{2^{n}}
 = (1+o(1))\ \sum_{i=k_n-\delta_n}^{k_n} a_i^{(n)}.\]
Moreover, since $k_n\leq M_{n-p}$, and since the sequence $(\delta_{n})_{n\geq 1}$ 
satisfies $\delta_{n}=o(k_n)$ and $\frac{k_n\sqrt{\ln k_n}}{\sqrt{n-p}} = o(\delta_n)$,
applying Lemma~\ref{lem:u_n} to the sequence $u_n = k_n$ gives us, asymptotically when $n$ tends to infinity,
\[\frac{\Lab_{n-p,k_n}}{2^{n-p}} = (1+o(1))\ \sum_{i=k_n-\delta_n}^{k_n} a_i^{(n-p)}.\]
Therefore,
\[\frac{\Lab_{n-p,k_n}}{\Lab_{n,k_n}}
= (2^{-p}+o(1))\ \frac{\sum_{i=k_n-\delta_n}^{k_n} a_i^{(n-p)}}
{\sum_{i=k_n-\delta_n}^{k_n} a_i^{(n)}}.
\]
We have
\begin{eqnarray*}
(k_n-\delta_n)^p\sum_{i=k_n-\delta_n}^{k_n} a_i^{(n-p)}
&\leq & \sum_{i=k_n-\delta_n}^{k_n} i^p a_i^{(n-p)}
= \sum_{i=k_n-\delta_n}^{k_n} a_i^{(n)}
= \sum_{i=k_n-\delta_n}^{k_n} i^p a_p^{(n-p)}
\leq k_n^p \sum_{p=k_n-\delta_n}^{k_n} a_p^{(n-p)},
\end{eqnarray*}
which implies
\[\frac{\Lab_{n-p, k_n}}{\Lab_{n,k_n}} 
\sim \frac1{(2k_n)^p}\quad \text{ when }n\to+\infty.\]
\vspace{\baselineskip}
{\bf (ii) Now assume that $\bs{M_{n-p} < k_n \leq M_n}$}.
Let $(\delta_n)_{n\geq 1}$ be an integer-valued sequence such that $\delta_n = o(k_n)$
and $\frac{k_n\sqrt{\ln k_n}}{\sqrt{n-p}} = o(\delta_n)$.
Let $(\eta_n)_{n\geq 1}$ be an integer-valued sequence such that $\eta_n = o(M_{n-p})$,
and $\sqrt{M_{n-p}\ln (k_n-M_{n-p})}=o(\eta_n)$.
Applying Lemma~\ref{lem:u_n} {\it (ii)} to the sequence $u_n = k_n$, we obtain
\[\frac{\Lab_{n-p,k_n}}{2^{n-p}} = (1+o(1))\ \sum_{i=M_{n-p}-\delta_n}^{\min\{M_{n-p}+\eta_n,k_n\}} a_i^{(n-p)}.\]
Moreover, since $\delta_n = o(k_n)$ and $\frac{k_n\sqrt{\ln k_n}}{\sqrt{n}} = o(\delta_n)$,
via Lemma~\ref{lem:u_n} {\it (i)},applied to the sequence $u_n=k_n$,
\[\frac{\Lab_{n,k_n}}{2^n} = (1+o(1))\ \sum_{i=k_n-\delta_n}^{k_n} a_i^{(n)}.\]
Let us remark, as above, that
\[(k_n-\delta_n)^p \sum_{i=k_n-\delta_n}^{k_n} a_i^{(n-p)} 
\leq \frac{\Lab_{n,k_n}}{2^n}
\leq k_n^p \sum_{i=k_n-\delta_n}^{k_n} a_i^{(n-p)}.\]
Moreover, since $k_n > M_{n-p}$,
using similar arguments as those developed to prove Lemma~\ref{lem:u_n} {\it (i)},
\[\sum_{i=k_n-\delta_n}^{k_n} a_i^{(n-p)} 
\sim \sum_{i=k_n-\delta_n}^{\min\{k_n,M_{n-p}+\eta_n\}} a_i^{(n-p)}
\sim \frac{\Lab_{n-p,k_n}}{2^{n-p}}.\]
Therefore, since $\delta_n = o(k_n)$, we get
\[\frac{\Lab_{n-p,k_n}}{\Lab_{n,k_n}} 
= (1+o(1))\ \frac1{(2k_n)^p},\]
which concludes the proof.
\end{proof}

\begin{lemma}\label{lem:kgrand}
Let $(k_n)_{n\geq 1}$ be a sequence of integers that tends to infinity when $n$ tends to infinity.
Let us assume that $k_{n} \geq M_n$ for large enough $n$, then, for all integer $p$,
asymptotically when $n$ tends to infinity,
\[\frac{\Lab_{n-p,k_n}}{\Lab_{n,k_n}} 
= (1+o(1))\ \left(\frac{\ln n}{2n}\right)^p.\]
\end{lemma}

\begin{proof}
By assumption, $k_n\geq M_n$, which implies $k_n\geq M_{n-p}$.
Let $(\delta_n)_{n\geq 1}$ be a sequence of integers such that 
$\delta_n = o(M_{n-p})$ and $\frac{M_n\sqrt{\ln M_n}}{\sqrt{n}} = o(\delta_{n+p})$.
Let $(\eta_n)_{n\geq 1}$ be another sequence of integers such that
 $\eta_n = o(M_{n-p})$, and $\sqrt{M_n\ln (k_n-M_n)}=o(\eta_{n+p})$.
We thus can apply Lemma~\ref{lem:u_n} {\it (ii)} to $u_n = k_n$
and conclude that, asymptotically when $n$ tends to infinity,
\[\frac{\Lab_{n-p,k_n}}{2^{n-p}} 
= (1+o(1))\ \sum_{i=M_{n-p}-\delta_n}^{\min\{M_{n-p}+\eta_n,k_n\}} a_i^{(n-p)}.\]
Moreover, since the sequence $(\delta_{n})_{n\geq 1}$ verifies $\delta_{n} = o(M_{n-p}) = o(M_n)$
and $\frac{M_n\sqrt{\ln M_n}}{\sqrt{n}} = o(\delta_{n+p}) = o(\delta_n)$,
and since the sequence $(\eta_n)_{n\geq 1}$ verifies $\eta_n = o(M_{n-p}) = o(M_n)$,
and $\sqrt{M_n\ln (k_n-M_n)}=o(\eta_{n+p}) = o(\eta_n)$, we have,
\[\frac{\Lab_{n,k_n}}{2^n} 
= (1+o(1))\ \sum_{i=M_n-\delta_n}^{\min\{M_n+\eta_n,k_n\}} a_i^{(n)}.\]
Let us note that
\begin{align*}
(M_n-\delta_n)^p \sum_{i=M_n-\delta_n}^{\min\{M_n+\eta_n,k_n\}} & a_i^{(n-p)}
\leq \frac{\Lab_{n,k_n}}{2^n} 
\leq (M_n+\eta_n)^p \sum_{i=M_n-\delta_n}^{\min\{M_n+\eta_n,k_n\}} a_i^{(n-p)}.
\end{align*}
Since $k_n\geq M_n \geq M_{n-p}$, 
via similar arguments to those developed for the proof of Lemma~\ref{lem:u_n} {\it (ii)}, we get
\[\sum_{i=M_n-\delta_n}^{\min\{M_n+\eta_n,k_n\}} a_i^{(n-p)}
\sim \sum_{i=M_n-\delta_n}^{\min\{M_{n-p}+\eta_n,k_n\}} a_i^{(n-p)}.\]
We thus have to compare
\[S_n = \sum_{i=M_n-\delta_n}^{\min\{M_{n-p}+\eta_n,k_n\}} a_i^{(n-p)}\]
and
\[T_n = \sum_{i=M_{n-p}-\delta_n}^{\min\{M_{n-p}+\eta_n,k_n\}} a_i^{(n-p)},\]
and to prove that those two sums are equivalent when $n$ tends to infinity. 
Decompose $S_n$ as follows:
\[S_n = T_n 
+ \sum_{i= \min\{M_{n-p}+\eta_n,k_n\}}^{\min\{M_n+\eta_n,k_n\}} a_i^{(n-p)} 
- \sum_{i= M_{n-p}-\delta_n}^{M_n-\delta_n} a_i^{(n-p)}.\]
Arguments from the proof of Lemma~\ref{lem:u_n} {\it (ii)} imply that the second
summand is negligible in front of the first. 
Let us assume that the third term is non-zero, i.e. $M_n-\delta_n > M_{n-p}-\delta_n$
(note that if this term is zero then $S_n \sim T_n$ is already proved).
Via Lemma~\ref{lem:croissance}, since $\frac{M_n}{M_{n-p}} = 1 + o(\frac1{M_n})$, we have
\begin{align*}
\sum_{i= M_{n-p}-\delta_n}^{M_n-\delta_n} a_i^{(n-p)}
&\leq (M_n-\delta_n - M_{n-p} + \delta_n) a_{M_{n-p}-\delta_n}^{(n-p)}
= o(1)\ a_{M_{n-p}-\delta_n}^{(n-p)} = o\left(a_{M_{n-p}}^{(n-p)}\right),
\end{align*}
in view of Lemma~\ref{lem:u_n} {\it (i)}.
Therefore, since $a_{M_{n-p}}^{(n-p)}\leq T_n$, we have $S_n \sim T_n$ when $n$ tends to infinity, which implies, 
since $\eta_n = o(M_n)$ and $\delta_n = o(M_n)$,
\[\frac{\Lab_{n-p, k_n}}{\Lab_{n,k_n}} 
= (1+o(1))\ \frac1{(2M_n)^p} = (1+o(1))\ \left(\frac{\ln n}{2n}\right)^p.\qedhere\]
\end{proof}

Finally, this fundamental technical part allows us to use Kozik's 
key ideas in order to describe
the probability distribution induced on Boolean functions,
in our two new models.

\section{Adjustment of Kozik's pattern language theory}\label{sec:patterns}

In 2008, Kozik~\cite{kozik08} introduced a quite effective way to study Boolean trees: 
he defined a notion of pattern that permits to easily classify and count large trees according to some constraints
on their structures.
Kozik applied this pattern theory to study the classical Catalan tree distribution.
We recall the definitions of patterns, illustrate them on examples 
and then extend Kozik's paper results in order to use them in our new models.
This part will extensively use Analytic Combinatorics (generating functions, symbolic methods, singularity analysis): 
we refer the reader to Flajolet \& Sedgewick's book~\cite{FS09} for an introduction to these methods.

\begin{definition}
\begin{enumerate}[(i)]
\item A {\bf pattern} is a binary tree with internal nodes labelled by
$\land$ or $\lor$ and with external nodes labelled by~$\bullet$ or~$\boxempty$.
Leaves labelled by~$\bullet$ are called {\bf pattern leaves} and leaves labelled by $\boxempty$
are called {\bf place-holders}. A {\bf pattern language} is a set of patterns.

\item Given a pattern language $L$ and a family of trees $\M$, we denote by $L[\M]$ the family of all trees
obtained by replacing every place-holder in an element from $L$ by a tree from $\M$. 

\item We say that $L$ is {\bf unambiguous} if, and only if, for any family $\mathcal{M}$ of trees,
any tree of $L[\mathcal M]$
can be built from a unique pattern from $L$ into which trees from $\mathcal{M}$ have been plugged.
\end{enumerate}
\end{definition}
The generating function of a pattern language $L$ is $\ell(x,y) = \sum_{d, p} L(d,p)x^dy^p$,
where $L(d,p)$ is the number of elements of $L$ with $d$ pattern leaves and $p$ place-holders.

\begin{definition}\label{df:composition}
We define the {\bf composition} of two pattern languages $L[P]$ to be the pattern language
of trees which are obtained by replacing every place-holder of a tree from $L$ by a tree from $P$.

Given an integer $i$ and a pattern $L$, the pattern $L^{(i)}$ is defined by the following recursion: $L^{(1)}= L$ and $L^{(i+1)} = L^{(i)}[L]$.
\end{definition}

\begin{definition}
A pattern language $L$ is {\bf sub-critical} for a family $\M$ if the generating function $m(z)$ of $\M$
has a square-root singularity $\tau$, and if $\ell(x,y)$ is analytic in some set
$\{(x,y) : |x|\leq \tau+\varepsilon, |y|\leq m(\tau)+\varepsilon\}$ for some positive $\varepsilon$. 
\end{definition}

\begin{definition}
Let $L$ be a unambiguous pattern language, $\mathcal M$ be a family of trees and $\Gamma$ a subset of $\{x_i\}_{i\geq 1}$,
which cardinality does not depend on $n$. Given an element of $L[\M]$, 
\begin{enumerate}[(i)]
\item the number of its $L$-{\bf repetitions} is the number of its $L$-pattern leaves minus the
 number of different variables that appear in the labelling of its $L$-pattern leaves.
\item the number of its $(L,\Gamma)$-{\bf restrictions} is the number of its $L$-pattern leaves
that are labelled by variables from $\Gamma$, plus the number of its $L$-repetitions.
\end{enumerate}
\end{definition}

\begin{definition}
Let $\I$ be the family of the trees with internal nodes labelled by a
connective and leaves without labelling, i.e. the family of tree-structures.
\end{definition}
The generating function of $\I$ satisfies $I(z) = z+2I(z)^2$, that implies $I(z) = (1-\sqrt{1-8z})/4$ and
thus its dominant singularity is $\nicefrac18$. Let $I_n$ be the $n$-th coefficient of $I(z)$.

\vspace{\baselineskip}
We can, for example, define the unambiguous pattern language $N$ by induction as follows: 
$N = \bullet | N\lor N | N\land \boxempty$, meaning that a pattern from  $N$ is either a single
pattern leaf, or a tree rooted by $\lor$ which two sub-trees are patterns from $N$, 
or a tree rooted by $\land$ which left sub-tree is a pattern from $N$ and which right sub-tree 
is a place-holder. An element of $N$ is represented in Fig~\ref{fig:patternN}.
Its generating function verifies $n(x,y) = x + n(x,y)^2 + yn(x,y)$
and is equal to $n(x,y) = \frac12(1-y-\sqrt{(1-y)^2-4x})$. It is thus sub-critical for $\mathcal{I}$.

\begin{figure}
\centering
\includegraphics[width=.5\textwidth]{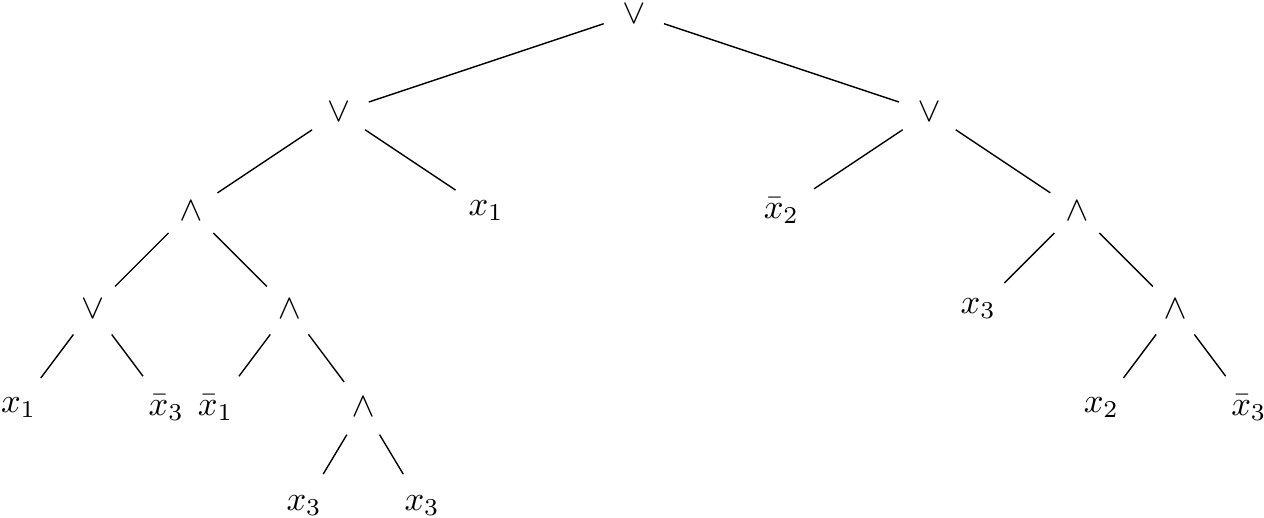}
\caption{The tree computes the function $x_1 \lor \lnot x_2$.}
\label{fig:ex_restrictions}
\end{figure}

\begin{figure}
\centering
\includegraphics[width=.4\textwidth]{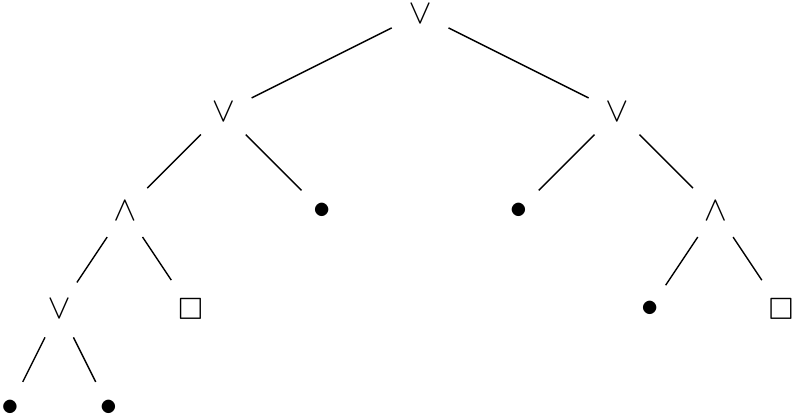}
\caption{The pattern is an element of the pattern language $N$.}
\label{fig:patternN}
\end{figure}

The tree depicted in Fig.~\ref{fig:ex_restrictions} is built from the pattern of Fig.~\ref{fig:patternN}.
It has 5 $N$-pattern leaves, 2 $N$-repetitions and 4 $(N,\{x_1,x_2\})$-restrictions.
It is also built from the pattern of Fig.~\ref{fig:patternN[N]} and has 2 $N[N]$-pattern leaves,
and 2 $(N[N],\{x_1,x_2\})$-restrictions.

\begin{figure}
\centering
\includegraphics[width=.6\textwidth]{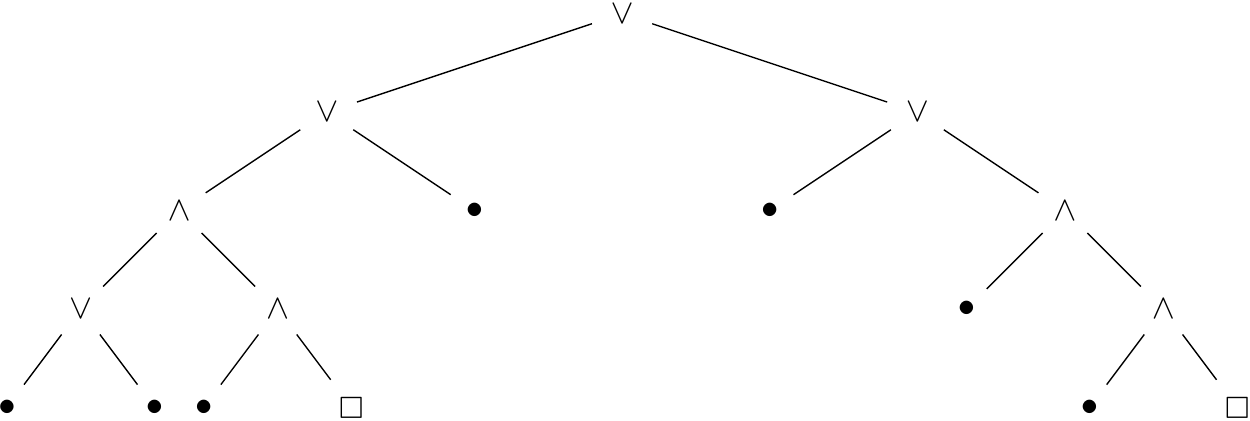}
\caption{The pattern is an element of the pattern language $N[N]$.}
\label{fig:patternN[N]}
\end{figure}

\vspace{\baselineskip}
The following key lemma is a generalization of the corresponding lemma of Kozik~\cite[Lemma 3.8]{kozik08}.
\begin{lemma}\label{lem:Jakub}
Let $L$ be an unambiguous pattern, sub-critical for the tree-structures family $\mathcal I$. 
Let $r$ be a fixed positive integer.
\begin{itemize}
\item[$\G$]
Let $A^{[r]}_{n}$ (resp. $A^{[\geq r]}_{n}$) 
be the number of labelled (with at most $k_n$ variables) trees of $L[\mathcal I]$
of size~$n$ and with $r$ $L$-repetitions (resp. at least $r$ L-repetitions).
\item [$\E$]
Let $A^{[r]}_{n}$ (resp. $A^{[\geq r]}_{n}$) 
be the number of \emph{equivalence classes} of labelled (with at most $k_n$ variables) trees of $L[\mathcal I]$
of size $n$ and with $r$ $L$-repetitions (resp. at least $r$ L-repetitions).
\end{itemize}

Then, asymptotically when $n$ tends to infinity, in both models,
\[\frac{A^{[r]}_{n}}{A_{n}} = \mathcal{O}\left(\rat_n^r\right) \hspace*{1cm}\text{ and }\hspace*{1cm}
\frac{A^{[\geq r]}_{n}}{A_{n}} = \mathcal{O}\left(\rat_n^r\right).\]
\end{lemma}

\begin{proof}
First recall that $A_n = I_n \cdot \Lab_{n,k_n}$ in both models.

\noindent{\bf Model $\G$.}
The number of labelled trees of $L[\mathcal I]$ of size $n$ and with at least $r$ $L$-repetitions is given by:
\[A_n^{[\geq r]} = \sum_{d=r+1}^{n} I_n(d) \cdot \Lab(n,k_n,d,r),\]
where $I_n(d)$ is the number of tree-structures with $d$ $L$-pattern leaves (among the $n$ number of leaves) and
$\Lab(n,k_n,d,r)$ corresponds to the number of leaf-labellings of these trees giving at least $r$ $L$-repetitions.
The following enumeration contains some multi-counting and we therefore get an upper bound:
\[\Lab(n,k_n,d,r) \leq 2^n \cdot  \sum_{j=1}^{r}\binom{d}{r+j} {r+j \brace j} k_n(k_n-1)\cdots(k_n-j+1) k_n^{n-r-j}.\]
The factor $2^n$ corresponds to the polarity of each leaf
(whether the literal is positive or negative);
the index $j$ stands for the number of different variables involved in the $r$ repetitions;
the binomial factor corresponds to the choices of the pattern leaves that are involved in the $r$ repetitions;
the Stirling number corresponds to the partition of the $r+j$ leaves into $j$ parts;
the factor $k_n(k_n-1)\cdots (k_n-j+1)$ stand for the choice of the repeated variables, from left to right;
finally, the factor $k_n^{n-r-j}$ corresponds to the choices of the variables assigned to all remaining leaves.
We have
\[\Lab(n,k_n,d,r) \leq 2^n k_n^{n-r} \cdot  \sum_{j=1}^{r}\binom{d}{r+j} {r+j \brace j},\]
in other terms,
\[\Lab(n,k_n,d,r) \leq 2^r \Lab_{n-r,k_n} \cdot  \sum_{j=1}^{r}\binom{d}{r+j} {r+j \brace j},\]
since $\Lab_{n,m} = (2m)^n$ (in model $\G$), and
\begin{equation}\label{eq:truc}
A_{n}^{[\geq r]} 
\leq 2^r\cdot \Lab_{n-r,k_n} \sum_{j=1}^{r}  {r+j \brace j} \sum_{d=r+j}^n I_n(d) \binom{d}{r+j}.
\end{equation}
Let $\ell(x,y)$ be the generating function of the pattern $L$. 
Note that $\frac{x^p}{p!} \partial_1^p \ell$ corresponds to pointing $p$ distinct
pattern leaves (without order) in the $L$-patterns 
(where $\partial_1$ stands for the derivative according to the first coordinate).
Then, for all $p\geq 0$,
\[\frac{z^p}{p!} \partial_1^p \ell(z,I(z)) = \sum_{n=1}^{\infty}\sum_{d=p}^{\infty} I_n(d)\binom{d}{p}z^n.\]
Thus, 
\[\frac{A_{n}^{[\geq r]}}{A_{n}} \leq 
\frac{2^r \Lab_{n-r,k_n}}{\Lab_{n,k_n}}
\sum_{j=1}^r {r+j \brace j} \frac{[z^n]z^{r+j}\partial_1^{r+j}\ell (z,I(z))}{[z^n]I(z)}.\]
Since $\partial_1^{r+j}\ell(z,I(z))$ and $I(z)$
have the same dominant singularity because of the sub-criticality of the pattern $L$ according to $\I$,
the previous sum tends to a constant (because $r$ is fixed)
when $n$ tends to infinity and so
we conclude, using Propositions~\ref{prop:labelling2} and~\ref{prop:labelling}:
\[\frac{A_{n}^{[r]}}{A_{n}} 
\leq \frac{A_{n}^{[\geq r]}}{A_{n}} 
= \mathcal{O}\left(\frac{\Lab_{n-r, k_n}}{\Lab_{n, k_n}} \right)
= \mathcal{O}\left(\rat_n^r\right).\]

\vspace{\baselineskip}
\noindent{\bf Model $\E$.}
The number of equivalence classes of labelled trees of $L[\mathcal I]$ of size $n$ and with at least $r$ $L$-repetitions is given by:
\[A_n^{[\geq r]} = \sum_{d=r+1}^{n} I_n(d)\cdot \Lab(n,k_n,d,r),\]
where $I_n(d)$ is the number of tree-structures with $d$ $L$-pattern leaves and
$\Lab(n,k_n,d,r)$ corresponds to the number of leaf-labellings of these trees giving at least $r$ $L$-repetitions.
The following enumeration contains some multi-counting and we therefore get an upper bound:
\[\Lab(n,k_n,d,r) \leq 2^n \cdot  \sum_{j=1}^{r}\binom{d}{r+j} {r+j \brace j} \frac{\Lab_{n-r,k_{n}}}{2^{n-r}}.\]
The factor $2^n$ corresponds to the polarity of each leaf
(whether the literal is positive or negative);
the index $j$ stands for the number of different variables involved in the $r$ repetitions;
the binomial factor corresponds to the choices of the pattern leaves that are involved in the $r$ repetitions;
the Stirling number corresponds to the partition of $r+j$ leaves into $j$ parts;
finally, the factor $\Lab_{n-r,k_n}$ corresponds to the rest of the partition.
Therefore,
\[A_{n}^{[\geq r]} 
\leq 2^r \cdot \Lab_{n-r,k_{n}} \sum_{j=1}^{r}  {r+j \brace j} \sum_{d=r+j}^{n} I_n(d) \binom{d}{r+j}.\]
Applying the same reasoning as for model $\G$ starting from
Equation~\eqref{eq:truc} permits to conclude the proof.
\end{proof}

We have finally adapted Kozik's theory in order to apply it in the new contexts.
Since we have extended the pattern theory, we are able to use in the following the same 
key-ideas to describe the probability distributions we are interested in.

\section{Behaviour of the probability distribution}\label{sec:proba}

Once we have adapted the pattern theory to our model and proved the central Lemma~\ref{lem:Jakub},
we are ready to prove our main results, namely Theorems~\ref{thm:thetaG} and~\ref{thm:thetaE}.
A first step consists to understand the asymptotic behaviour of $\mathbb P_n^{\G}(\true)$ and $\mathbb P_n^{\E}\langle \true\rangle$.

It is natural to focus on this ``simple'' function before considering a general class $\langle f \rangle$;
and it happens to be essential for the continuation of the study.
In addition, the methods used to study tautologies (mainly pattern theory)
will also be the core of the proof for a general function (model~$\G$) or
a general equivalence class (model~$\E$).

First, let us introduce some measure in the context of Boolean expressions.
Given a family $\mathcal G$ of \ao trees (resp. equivalence classes of \ao trees),
we define its {\bf ratio} $\mu_n(\mathcal G)$ as follows: let $G_n$ be the number of elements of $\mathcal G$ of size~$n$, 
\[\mu_n(\mathcal G) := \frac{G_n}{A_n}.\]

\subsection{Tautologies}

First note that $\true$ is the unique element of its equivalence class $\langle \true\rangle$.

A {\bf tautology} is an \ao tree that represents the Boolean function $\true$.
By symmetry, the functions $\true$ and $\false$ have the same probability in both models.
Let $\mathcal T$ be the family of tautologies.
In this part, we prove that the probability of $\true$ is asymptotically equal
to the ratio of a simple subset of tautologies.

\begin{definition}[cf.~Fig.~\ref{fig:st}]
A {\bf simple tautology} is an \ao tree that contains two leaves
labelled by a variable $x$ and its negation $\bar{x}$ and
such that all internal nodes from the root to both these leaves are labelled by $\lor$-connectives.
We denote by $\mathcal S$ the family of simple tautologies.
\end{definition}

\begin{figure}
\centering
\begin{tikzpicture}[style={level distance=1cm},level/.style={sibling distance=40mm/#1}, scale=0.8]
 \node [circle] (z){$\vee$}
 child {node [circle] (a){$\vee$}
   	child{node [circle] (b) {$\vee$} 
 		child{node [circle] (c) {$\cdots$} }
 		child{node [circle] (d) {$\vee$} 
 			child{node [circle] (x) {$x$}}
 			child{node [circle] (y) {$\cdots$}}
 			}
 		}
 	child{node [circle] (e) {$\cdots$} }
 	}
 child {node [circle] (f) {$\vee$}
 	child{node [circle] (g) {$\cdots$} }
 	child{node [circle] (h) {$\vee$}
 		child{node [circle] (i) {$\bar{x}$} }
 		child{node [circle] (l) {$\cdots$} }
 		}
 	}
 ;
\end{tikzpicture}
\caption{A simple tautology.}
\label{fig:st}
\end{figure}

\begin{proposition}\label{prop:tauto}
The ratio of simple tautologies verifies
\[\mu_{n}(\mathcal S) \sim \frac32 \cdot \rat_n, \text{ when $n$ tends to infinity.}\]
Moreover, asymptotically when $n$ tends to infinity, almost all tautologies are simple tautologies, meaning that
\[\mu_n(\mathcal T) \sim \mu_n(\mathcal S), \text{ when $n$ tends to infinity.}\]
\end{proposition}

\begin{proof}
The proof is divided in two steps. The first one is dedicated to the
computation of the ratio $\mu_{n}(\mathcal S)$.
The second part of the proof shows that almost all tautologies are simple tautologies.

Let us consider the non-ambiguous pattern language $M = \bullet | M \lor M | \boxempty\land\; \boxempty$. 
%It is sub-critical for $\mathcal I$.
Remark that a tree such that two $M$-pattern leaves are labelled by a variable and its negation,
is a simple tautology. The generating function of $M$ is $m(x,y) = \frac12(1-\sqrt{1-4(x+y^2)})$. It is sub-critical for $\mathcal{I}$. 

The generating function $\tilde{I}(z) = \frac12\nicefrac{\partial^2}{\partial x^2}(m(xz,I(z))_{|x=1}$
enumerates \ao trees with two marked distinct leaves linked to the root by $\mathsf{or}$-nodes.
Therefore, $DC_n = \tilde{I}_n \cdot \Lab_{n-1, k_n}$ is the number of simple tautologies where
simple tautologies realized by a unique pair of leaves are counted once,
those that are realized by two pairs of leaves are counted twice, and so on.
We have
\[\frac{DC_n}{A_n} = \frac{\tilde{I}_{n}\cdot \Lab_{n-1,k_n}}{I_{n} \cdot \Lab_{n,k_n}},\]
and using a consequence of~\cite[Theorem VII.8]{FS09} (cf. a detailed proof in~\cite{GK12}):
\[\lim_{n\to\infty}\frac{\tilde{I}_{n}}{I_{n}} = \lim_{z\to\frac18}\frac{\tilde{I}'(z)}{I'(z)}.\]
Note that
\[\tilde I(z) = \frac{z^2}{\left(1-4(z+I(z)^2)\right)^{\nicefrac32}},\]
and thus,
\[\frac{\tilde I'(z)}{I(z)} = 
\frac{2z}{\left(1-4(z+I(z)^2)\right)^{\nicefrac32}}
+ \frac{(1+ 2I'(z)I(z))}{I'(z)}\ \frac{6 z^2}{\left(1-4(z+I(z)^2)\right)^{\nicefrac52}}.\]
Note that, when $z\to \nicefrac18$, $I'(z) \to+\infty$. Moreover, $I(\nicefrac18)=\nicefrac14$. Thus,
\[\frac{\tilde I'(z)}{I(z)} 
\sim \frac{\nicefrac{3}{8^2}}{\left(1-4(\nicefrac18+\nicefrac1{16})\right)^{\nicefrac52}}
= \frac32 \quad\text{ when }z\to\frac18.\]
Thus, we get the upper bound $\nicefrac32\cdot \rat_n$ for the ratio of simple tautologies:
it remains to deal with the double-counting in order to compute a lower bound.

In $DC_n$, simple tautologies realized by a unique pair of leaves are counted once,
those that are realized by two pairs of leaves are counted twice, and so on.
Let us denote by $ST^i_n$ the number of simple tautologies counted at least $i$ times in $DC_n$: 
we have $DC_n = \sum_{i\geq 1} ST_n^{(i)}$.

Our aim is to remove from $DC_n$ the tautologies that have been over-counted. 
Therefore, we count simple tautologies realized by three $M$-pattern leaves labelled
by $\alpha/\alpha/\bar{\alpha}$ where $\alpha$ is a literal, and the tautologies
realized by four $M$-pattern leaves labelled by $\alpha/\bar{\alpha}/\beta/\bar{\beta}$ where $\alpha$ and $\beta$ are two different literals. Let us denote by
\[I_3(z) = \frac1{3!}\frac{\partial^3}{\partial x^3} m(xz,I(z))_{|x=1}\]
the generating function of tree-structures in which three $M$-pattern leaves have been pointed and
\[I_4(z) = \frac1{4!}\frac{\partial^4}{\partial x^4} m(xz,I(z))_{|x=1}\]
the generating function of tree-structures in which four $M$-pattern leaves have been pointed.
Then, let
\[DC_n^{(3)} = {3\cdot \Lab_{n-2,k_n}[z^n]I_3(z) }\quad
\text{ and }
\quad DC_n^{(4)} = {3\cdot \Lab_{n-2,k_n}[z^n]I_4(z) }.\]
The integer $DC_n^{(3)}$ (resp. $DC_n^{(4)}$) counts (possibly with multiplicity) the trees in which three (resp. four) $M$-pattern leaves have been pointed, one of them labelled by a literal and the two others by its negation (resp. two of them labelled by two literals associated to two different variables and the two others by their negations).
Remark that a tree having six $M$-pattern leaves labelled by $\alpha/\alpha/\bar{\alpha}/\beta/\beta/\bar{\beta}$ is counted twice by $DC_n^{(3)}$ and four times by $DC_n^{(4)}$.

For all integer $i$, a simple tautology counted at least $i$ times by $DC_n$ is counted at least $(i-1)$ times by $DC_n^{(3)}+DC_n^{(4)}$. Therefore,
\[ST_n \geq DC_n - (DC_n^{(3)} + DC_n^{(4)}).\]
In view of Lemma~\ref{lem:Jakub},
\[\frac{DC_n^{(3)}}{T_n} \leq c_3\cdot \frac{\Lab_{n-2,k_n}}{\Lab_{n,k_n}}\quad
\text{ and }
\quad\frac{DC_n^{(4)}}{T_n} \leq c_4\cdot \frac{\Lab_{n-2,k_n}}{\Lab_{n,k_n}},\]
where $c_3$ and $c_4$ are positive constants.
Then, asymptotically when $n$ tends to infinity, in view of Propositions~\ref{prop:labelling2} and~\ref{prop:labelling}:
$\mu_n(\mathcal F) = \mu_n(DC) + o\left(\rat_n\right) \sim \nicefrac32\cdot \rat_n$.

\vspace{\baselineskip}
Let us now turn to the second part of the proof: asymptotically, almost all tautologies are simple tautologies.
Let us consider the pattern $N = \bullet | N \lor N | N\land \boxempty$.
This pattern is unambiguous, its generating function satisfies $n(x,y) = x + n(x,y)^2 + y\cdot n(x,y)$
and is thus equal to $\frac12(1-y-\sqrt{(1-y)^2-4x})$. Consequently,
 $N$ is sub-critical for the family $\mathcal{I}$ of tree-structures.

A tautology has at least one $N[N]$-repetition. 
Otherwise, we can assign all its $N$-pattern leaves to false and,
the whole tree computes false: impossible for a tautology.

Consider a tautology $t$ with exactly one $N[N]$-repetition. this repetition must be a $x|\bar{x}$
repetition and must occur among the $N$-pattern leaves, using the same kind of argument than above.

Then, let us assume that there is an $\land$-node denoted by $\nu$ between the $N$-pattern
leaf $x$ and the root of the tree. This node $\nu$ has a left sub-tree $t_1$ and a right sub-tree $t_2$. 
%Assume that the leaf $x$ appears in $t_1$. Then, one can assign all the
%$N$-pattern leaves of $t_2$ (which are $N[N]$-pattern leaves of $t$) to false, since there
%is no more repetition among the $N[N]$-pattern leaves of $t$. Also assign all the pattern
%leaves of $t$ minus the sub-tree rooted at $\nu$ to false. Then, we can see that $t$ 
%computes false: impossible. We have thus shown that $t$ is a simple tautology.
Necessarily the leaf $x$ appears in $t_1$. Then, one can assign all the
$N$-pattern leaves of $t_2$ (which are $N[N]$-pattern leaves of $t$) to false, since there
is no more repetition among the $N[N]$-pattern leaves of $t$. Also assign all the $N[N]$-pattern
leaves of $t$ minus the sub-tree rooted at $\nu$ to false. Then, we can see that $t$ 
computes false: impossible. We have thus shown that $t$ is a simple tautology.

In a nutshell, tautologies with exactly one $N[N]$-repetition are simple tautologies,
a tautology must have at least one $N[N]$-repetition and, thanks to Lemma~\ref{lem:Jakub},
tautologies with more than one $N[N]$-repetitions have a ratio of order $o\left(\rat_n\right)$, 
which is negligible in front of the ratio of simple tautologies.
\end{proof}

The latter proposition gives us for free the proof for the satisfiability problem. 
In fact, both dualities between the two connectives and positive and negative literals
transform expressions computing $\true$ to expressions computing $\false$,
which implies $\IP_n^{\G}(\false) = \nicefrac{3}{2}\cdot\rat_n$.
Moreover, the only expressions that are not satisfiable compute the function $\false$ and 
$\IP_n^{\G}(\false) = \nicefrac{3}{2}\cdot\rat_n$ tends to $0$ as $n$ tends to infinity, 
which proves Corollary~\ref{thm:satis}.

\subsection{Proofs of Theorems~\ref{thm:thetaG} and~\ref{thm:thetaE}}\label{sec:main}
%
%With similar arguments than those used for tautologies,
%it is possible to prove that the probability of the class of projections (i.e. $(x_i)_{i\geq 1} \mapsto x_j$)
%is equivalent to $\nicefrac58 \cdot \rat_n$, when $n$ tends to $+\infty$.

This last section is devoted to the general result, i.e. to the study of
the behaviour of $\mathbb P^{\G}_n(f)$ and $\mathbb P^{\E}_n\langle f\rangle$ for all non constant Boolean function $f$.
The main idea of this part is that, roughly speaking,
\emph{a typical tree computing a Boolean function $f$ is a minimal tree
of $f$ into which a single large tree has been plugged}.
%
%\begin{proof}[sketch]
%Our aim is to describe the asymptotic behaviour of $\IP_n\langle f\rangle$,
%for a given class of Boolean functions $\langle f\rangle$.
%\begin{itemize}
%\item We first define several notions of \emph{expansions} of a tree: the idea is to replace
%in a tree, a sub-tree $S$
%by $T \land S$, where $T$ is chosen such that the expanded tree still computes the same function.
%
%\item The ratio of minimal trees of $\langle f\rangle$ expanded once is of the order of $\rat_n^{R(f)+1}$.
%
%\item The ratio of trees computing a function from $\langle f\rangle$ is
%asymptotically equal to the ratio of minimal trees expanded once.
%\end{itemize}
%The most technical part of the proof is the last one, because we need a precise upper bound of $\IP_n\langle f\rangle$.
%But the ideas are more or less the same as those developed for the class $\langle \true\rangle$.\hfill\qed
%\end{proof}

In the following, $f$ (resp. $\langle f\rangle$) is fixed,we denote by $r=L(f)$ its complexity,
and by $\Gamma_f$ the set of the essential variables of $f$.
We also fix~$t$ to be an \ao tree computing $f$. 

Moreover, we will need the folowing patterns:
\[N = \bullet | N\lor N | N\land \boxempty,\]
\[P = \bullet | P\lor \boxempty | P\land P,\]
and (see Definition~\ref{df:composition} where the composition of patterns is defined)
\[R = N^{(r+1)}[N\oplus P]\quad \text{ and }\quad \bar{R} = N^{(r+1)}[(N\oplus P)^2],\]
where the language $N\oplus P$ is defined such that the $N\oplus P$-pattern
leaves of a tree are its $N$-pattern leaves plus its $P$ pattern leaves.
It is proved in~\cite{kozik08} that this pattern language is indeed
non-ambiguous and sub-critical for $\mathcal I$ if $N$ and $P$ are
non-ambiguous and sub-critical for $\mathcal I$.

We have already noticed that assigning all $N$-pattern leaves of a Boolean tree to false 
make the whole tree calculate false. 
The pattern $P$ has the dual property that: 
assigning all the $P$-patterns leaves of a tree
to true make the whole tree calculate true.
This is why these two patterns are so useful in the proof of our main result.

\begin{proposition}
A tree $t$ computing $f$ (define $r:=L(f)$) 
with at least one leaf on the $(r+2)^{\text{th}}$ level of the $R$-pattern
must have at least $r+1$ $(R,\Gamma_f)$-restrictions.
\end{proposition}

\begin{proof}
Let us assume that $t$ computes $f$, 
and has at least one leaf on the $(r+2)^{\text{th}}$
level of the $R$ pattern but has less than $r$ $R$-repetitions.
Let $i$ be the smallest integer (smaller than $r+2$) such that the number
of $(N^{(i)},\Gamma_f)$-restrictions is equal to the number of $(N^{(i-1)},\Gamma_f)$-restrictions.

There must be either a repetition or an essential variable in the first level:
if there is none, then we can assign all the $N$ pattern leaves to $\false$
and this operation does not changes the represented function.
This function is then the constant function $\false$, which is impossible;
so $i\leq r+1$.

{\bf First case:} 
Let us assume that there are strictly less than $r$ $(N^{(i)},\Gamma_f)$-restrictions.
There is no repetition and no essential variable in the pattern leaves at level $i$. Therefore, we can assign them all to $\false$ and make the place-holders of the level $i-1$ compute $\false$. Let us replace those place-holders by $\false$ in the tree. Furthermore, replace by $\false$ all the non-essential remaining variables. And simplify the obtained tree to simplify all the constant leaves $\false$ and $\true$. We obtain a tree $t^{\star}$, which still computes $f$, and whose leaves are all former $N^{(i-1)}$ pattern leaves of $t$ labelled by essential variables. The tree $t^{\star}$ therefore contains strictly less than $r$ leaves, which is impossible since the complexity of $f$ is $r$.

{\bf Second Case:}
Let us assume that $t$ has exactly $r$ $(N^{(i)},\Gamma_f)$-restrictions.
Since $i\leq r+1$, there is no restriction in the place-holders of the level $r+2$. Therefore, we can replace the place-holders by wild-cards $\star$, which means that those wild-cards can be evaluated to $\true$ or $\false$ independently from each other and without changing the function computed by $t$. We can also replace the remaining leaves labelled by non-essential and non-repeated variables by such wild-cards. 

We simplify those wild-cards. Such a simplification has to delete at least one non-wild-card leaf. If we deleted a non-repeated essential variable, then the tree $t^{\star}$ does not depend on this essential variable and computes $f$: this is impossible. Thus, we deleted a repetition: $t^{\star}$ has strictly less than $R(f)$ repetitions and computes $f$. It is impossible.
\hfill\qed
\end{proof}

Remark that in Lemma~\ref{lem:Jakub}, we only count repetitions and not restrictions
as it was done in the original lemma by Kozik. Though, we will need to consider essential variables
and the following lemma permits to handle them. An {\bf expansion} of a tree $t$ is a tree obtained
by replacing a sub-tree $s$ of $t$ by $s\diamond t_e$ (or $t_e\diamond s$) where $\diamond\in\{\land,\lor\}$.

\begin{lemma}
\label{lem:restrictions}
Let $L$ be an unambiguous pattern, sub-critical for $\mc{I}$. 
Let $f$ be a fixed Boolean function, $\Gamma_f$ the set of its essential variables, 
and $\mc{M}_f$ the set of minimal trees computing $f$. 
Let $\mc{E}$ be the family of trees obtained by expanding once a tree of $\mc{M}_f$ 
by trees having exactly $p$ $(L,\Gamma_f)$-restrictions. Then, there exists a constant $\alpha^\G >0$ (resp. $\alpha^\E >0$) such that
\[\mu_{n}(\mc{E}) \sim \alpha^\G \cdot \rat_n^{L(f)+p} \text{ in model }\G,\]
resp.
\[\mu_{n}(\mc{E}) \sim \alpha^\E \cdot  \rat_n^{R\langle f\rangle +p} \text{ in model }\E.\]
\end{lemma}

\begin{proof}
Let $E_n$ be the number of (resp. equivalence classes of) trees of size $n$ in $\mc{E}$. 
We will denote by $i$ the number of leaves that are involved in 
the $p$ $(L,\Gamma_f)$-restrictions of the expansion tree: $p+1\leq i\leq 2p$. 
Let $\gamma_f$ be the cardinal of~$\Gamma_f$.

{\bf In the model $\boldsymbol{\G}$}, for all large enough $n$,
\begin{align*}
&\mu_n(\mc{E}) 
= \frac{E_n}{A_n}
\leq \cst_f 
\sum_{i=p+1}^{2p}[z^{n-L(f)}]\frac{\partial ^i}{i!\partial x^i}\left(\ell(xz,I(z))\right)_{|x=1} 
\frac{(2\gamma_f)^p (2(k_n-\gamma_f))^{n-L(f)-p}}{I_n (2k_n)^n},
\end{align*}
where $\cst_f = 2 L(f) \cdot |\M_f|$ is an upper bound for the different
places in a minimal tree of $f$ where an expansion can be plugged in.
Since $L$ is sub-critical for $\mc{I}$, there exists a positive constant $\alpha$ such that
\[\sum_{i=p+1}^{2p}\frac{[z^{n-L(f)}]\nicefrac{\partial ^i}{i!\partial x^i}\left(\ell(xz,I(z))\right)_{|x=1}}{I_n} 
\sim \alpha\cdot \frac{I_{n-L(f)}}{I_n} \sim \alpha\left(\frac18\right)^{L(f)} >0\]
asymptotically when $n$ tends to infinity. 
Therefore, in view of Section~\ref{sec:technical},
we have
\[\mu_{n}(\mc{E}) \sim \alpha \cdot \rat_n^{L(f) + p}.\]

{\bf In the model $\boldsymbol{\E}$}, we have, with the same reasoning:
\[
\mu_n\langle \mc{E} \rangle 
= \frac{E_n}{A_n}
\leq \cst_f 
\sum_{i=p+1}^{2p}[z^{n-L(f)}]\frac{\partial ^i}{i!\partial x^i}\left(\ell(xz,I(z))\right)_{|x=1} 
\frac{2^{p+R\langle f\rangle} \cdot \Lab_{n-p-R\langle f\rangle,k_n}}{I_n \cdot \Lab_{n,k_n}},
\]
from which we state the same conclusion as for the model $\G$.
\end{proof}

Consider the family $\mc{E}$ of trees obtained by replacing a sub-tree $s$ by $s\land t_e$
where $t_e$ is a simple tautology into a minimal tree of $f$. 
%Let us denote by $E_n$ the number of such trees of size $n$.
Since a simple tautology has at least one $S$-repetition,
thanks to Lemma~\ref{lem:restrictions}, there exists two positive constants $\alpha^\G$ and $\alpha^\E$
 such that
\[\mu_n^\G(\mc{E})  \sim  \alpha^\G \cdot \rat_n^{L(f)+1} \text{ in model }\G,\]
and
\[\mu_n^\E \langle \mc{E} \rangle  \sim \alpha^\E \cdot \rat_n^{R\langle f\rangle+1} \text{ in model }\E.\]

Thanks to Lemma~\ref{lem:Jakub}, we know that terms computing $f$ with more than $R(f)+2$
repetitions are negligible in front of the above family. Therefore, since trees with no leaf on
the $(r+2)^{\text{th}}$ level are negligible, we have proved
weaker versions of Theorems~\ref{thm:thetaG} and~\ref{thm:thetaE},
where the equivalent for the probabilities is 
replaced by an upper and a lower bounds of the same order.
The rest of the proofs consists in sharpening both bounds.

%In fact, we can show a more precise result:
%\begin{theorem}
%Let $f$ be a fixed Boolean function, then, asymptotically when $n$ tends to infinity,
%\[\mathbb{P}_{n}\langle f\rangle \sim \lambda_{\langle f\rangle}\rat_n^{R(f)+1},\]
%where $\lambda_{\langle f\rangle}$ is a positive constant.
%\end{theorem}

The key point of the proof of Theorems~\ref{thm:thetaG} and~\ref{thm:thetaE}
is that a typical tree computing a function $f$ is a minimal tree of this function
which has been expanded once. In the following, we will only consider two different expansions:
\begin{definition}[cf. Figure~\ref{fig:exp}]
Recall that an {\bf expansion} of a tree $t$ is a tree obtained by replacing a sub-tree $s$ of $t$ by $s\diamond t_e$ (or $t_e\diamond s$) where $\diamond\in\{\land,\lor\}$.

An expansion is a {\bf T-expansion} if the expansion tree $t_e$ is a simple tautology and the connective $\diamond$ is $\land$ (or a simple contradiction and the connective $\diamond$ is $\lor$).

An expansion is a {\bf X-expansion} if the expansion tree $t_e$ has a leaf linked to the root by a $\land$-path (resp. a $\lor$-path) and the $\diamond$ connective is a $\lor$ (resp. $\land$).
\end{definition}

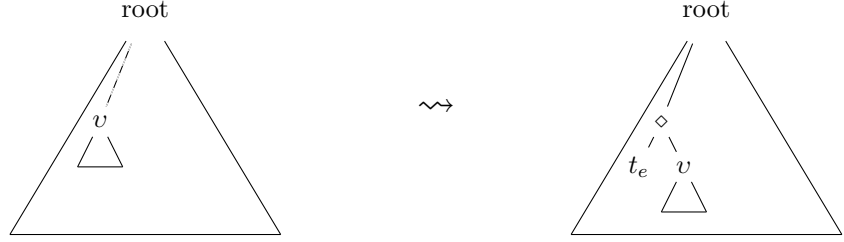
\begin{figure}[t]
\begin{center}
\begin{tabular}{m{5cm}m{0.5cm}m{5cm}}
\begin{tikzpicture}[level 1/.style={sibling distance=2cm}, level 2/.style={sibling distance=1cm,level distance=1cm}, scale=0.6]
\node [circle] (z){root}
	child[level distance=5cm]
	child[level distance=2.5cm] {node (x){$\upsilon$}
		child[level distance=1cm]
		child[level distance=1cm]
		}
	child[level distance=2.5cm] {node (m){} edge from parent[draw=none]}
	child[level distance=5cm]
;
\draw (z-1)--(z-4);
\draw[dashed,white] (z)--(x);
\draw (x-1)--(x-2);
\end{tikzpicture}
&
{\Large $\leadsto$}
&
\hspace*{1cm}
\begin{tikzpicture}[level 1/.style={sibling distance=2cm}, level 2/.style={sibling distance=1cm,level distance=1cm}, scale=0.6]
\node [circle] (z){root}
	child[level distance=5cm]
	child[level distance=2.5cm] {node (diamant){$\diamond$}
		child {node (bla){$t_e$}}
		child {node (x){$\upsilon$}
			child[level distance=1cm]
			child[level distance=1cm]
			}
		}
	child[level distance=2.5cm] {node (m){} edge from parent[draw=none]}
	child[level distance=5cm]
;
\draw (z-1)--(z-4);
\draw[dashed,white] (z)--(x);
\draw (x-1)--(x-2);
\end{tikzpicture}
\end{tabular}
\end{center}
\caption{An expansion at node $\upsilon$. Note that the expansion tree $t_e$ could have been on the right size of the $\diamond$-connective instead of its left side.}
\label{fig:exp}
\end{figure}

%\begin{lemma}
%The ratio of the (resp. equivalence class of) minimal trees of $f$ expanded once satisfies,
%asymptotically when $n$ tends to infinity
%\[\mu_{n}(E[\mc{M}_f]) = \alpha\cdot \rat_n^{R(f)+1} + o\left(\rat_n^{R(f)+1}\right).\]
%\end{lemma}
\begin{corollary}
The ratio of the (resp. equivalence class of) minimal trees of $f$ expanded once satisfies that
there exists two positive constant $\lambda_f$ and $\lambda_{\langle f\rangle}$ such that
asymptotically when $n$ tends to infinity:
\[\mu_{n}^\G (E[\mc{M}_f]) = \lambda_f\cdot \rat_n^{L(f)+1} + o\left(\rat_n^{L(f)+1}\right),\]

\[\mu_{n}^\E \langle E[\mc{M}_f]\rangle = \lambda_{\langle f\rangle}\cdot \rat_n^{R\langle f \rangle+1} + o\left(\rat_n^{R \langle f\rangle +1}\right).\]
\end{corollary}
This corollary is a direct consequence of Lemma~\ref{lem:restrictions}.
%
%
%%\begin{lemma}
%%The ratio of minimal trees of $f$ expanded at least twice verifies
%%\[\mu_n(E^{\geq 2}[\mc{M}_f] = o\left(\left(\frac{\ln n}{n}\right)^{R(f)+1}\right).\]
%%\end{lemma}

\begin{lemma}
Let $f$ be a fixed Boolean function and $\mc{M}_f$ the set of minimal trees of $f$.
\[\mathbb{P}^{\G}_n(f) \sim \mu^\G_n(E[\mc{M}_f]) \text{ when }n\to+\infty,\]
and
\[\mathbb{P}^{\E}_n\langle f\rangle \sim \mu_{n}^\E \langle E[\mc{M}_f]\rangle \text{ when }n\to+\infty.\]
\end{lemma}

\begin{proof}
Let $t$ be a tree computing $f$. 
Such a tree must have at least $R(f)+1$ $\bar{R}$-repetitions. 
Moreover, thanks to Lemma~\ref{lem:Jakub}, 
trees with at least $R(f)+2$ $\bar{R}$-repetitions are negligible. 
We will show that a tree with exactly $R(f)+1$ $\bar{R}$-repetitions is in fact a minimal tree expanded once.

The term $t$ must also have $R(f)+1$ $R$-repetitions and therefore, there is no additional repetition when
we consider the $(r+3)^{\text{th}}$ level of the $\bar{R}$-pattern.

Let $i$ be the first level such that the number of $(N^{(i)},\Gamma_f)$-restrictions is equal to the number of
$N^{(i-1)}$-restrictions. Since there must be a restriction on the first level, $i\leq r+1$.

\vspace{\baselineskip}
{\bf First Case: }Assume that an essential variable $\alpha$ appears on the pattern leaves of the $(r+3)^{\text{th}}$ level.  
Therefore, $t$ has at most $L(f)$ $(N^{(i)},\Gamma_f)$-restrictions. Let us replace the place-holders of the $(i-1)^{\text{th}}$ level by $\false$ and assign all the remaining non-essential variables to $\false$. Simplify the tree to obtain a new and/or tree denoted by $t^{\star}$. The leaves of this tree are former $N^{(i-1)}$-pattern leaves of $t$, labelled by essential variables and $t^{\star}$ still computes $f$. But the variable $\alpha$ is essential for $f$: thus it must still appear in the leaves of $t^{\star}$, and by deleting its occurrence in the leaves of the $(r+3)^{\text{th}}$ level, we deleted one repetition. Therefore, $t^{\star}$ has at most $L(f)-1$ leaves which is impossible!

\vspace{\baselineskip}
{\bf Second Case: }There is no essential variable among the the pattern leaves of the $(r+3)^{\text{th}}$ level. Since there is also no repetition at this level, we can replace the place-holders of the level $(r+3)$ to wild-cards. We also replace the remaining non essential and non-repeated variables by wild-cards. We then simplify the wild-cards and obtained a simplified tree $t^{\star}$, computing $f$, with no wild-cards and which leaves are former leaves of the trees $t$, essential or repeated. During the simplification process, we have deleted at least one of these leaves and therefore $t^{\star}$ has at most $L(f)$ leaves: it is a minimal tree of~$f$.

\vspace{\baselineskip}
Let us consider the following fact:
The lowest common ancestor of all the wild-cards in $t$ has
been suppressed during the simplification process.
Assume that this fact is false: then two wild-cards have been simplified independently during the simplification process, and thus, at least two essential or repeated variables have been deleted. The tree $t^{\star}$ has thus at most $L(f)-1$ leaves and computes $f$, which is impossible since $L(f)$ is the complexity of~$f$.
Let us denote by $t_e$ the sub-tree rooted at $\upsilon$ the lowest common ancestor of the wild-cards. Thus a typical tree computing $f$ is a minimal tree of~$f$ in which we have plugged a specific expansion tree $t_e$.
\end{proof}

\begin{lemma}
Let $t$ be a typical tree computing $f$. The expansion tree $t_e$ is either a simple tautology (or simple contradiction), or an $x$-expansion - i.e. a tree with one $\land$-leaf (resp. $\lor$-leaf) labelled by an essential variable of~$f$.
\end{lemma}

\begin{proof}
As shown in the former lemma, a typical tree computing $f$ is a minimal tree of $f$ on which has been plugged an expansion tree $t_e$.

{\bf First Case:} Let us assume that $t_e$ has no $(N\oplus P)$-repetition and no essential variable among its $(N\oplus P)$-pattern leaves. Then, we can replace $t_e$ by a wild-card and simplify this wild-card. This simplification suppresses at least one other leaf of the tree: the obtained tree is then smaller than the original minimal tree, and still computes $f$. It is impossible.

{\bf Second Case:} Let us assume that $t_e$ has at least two $((N\oplus P)^2,\Gamma_f)$-restrictions.
Thanks to Lemma~\ref{lem:restrictions}, this family of expanded trees is negligible.

{\bf Third Case:} Let us assume that $t_e$ has exactly one $((N\oplus P)^2,\Gamma_f)$-restriction.
Then it must be a $(N\oplus P,\Gamma_f)$-restriction (see First Case).
\begin{itemize}
\item if it is a repetition, than one can show that it must be a simple tautology or a simple contradiction.
\item if it is an essential variable, one can show that it must be an $X$-expansion.
\end{itemize}
\indent \hfill\qed
\end{proof}

\section{Conclusion}
In this paper, we have generalised the Catalan tree distribution on Boolean functions following two directions:
\begin{itemize}
\item letting the number of variables and the size of the Boolean trees tend to infinity together.
 It has allowed us to answer a fundamental satisfiability problem;
\item  the natural equivalence relation on Boolean trees and functions that we have introduced
exhibits a very interesting threshold/saturation phenomenon for which we have no intuitive explanation up to now.
\end{itemize}
It is interesting to see that these two models can be analysed with very similar methods, 
namely, the ones used in the literature to study the classical Catalan tree model: 
Analytic Combinatorics and Kozik's pattern theory.
The key idea that permitted to generalise those methods to our two new models was to dissociate 
the shapes of the trees and their leaf-labelling.

We strongly believe that our methods could be generalised further, 
for example to other logical systems (as the implication model, see e.g.~\cite{FGGG12,GK12}), 
or to non-binary or non-planar uniform trees (see~\cite{GGKM15}).
Our confidence rely on the fact that those models, in the $(k_n)_{n\geq 1}$
constant case, can be analysed with analytic combinatorics and pattern theory
(or tools based on the same key ideas) as well,
and we have shown here how to generalise those methods to a more general sequence $(k_n)_{n\geq 1}$.

A more challenging generalisation would be to consider different probability distributions on binary plane trees. 
For example, in view of~\cite{FGG09,CGM14} we conjecture that the random binary search tree of size $n$,
labelled with $(k_n)_{n\geq 1}$ variables defines a very interesting satisfiability problem, 
with a phase transition \emph{{\`a} la} $K$--{\sc sat}. 
It would be very interesting (but, we expect, non trivial) to prove such a conjecture.
Even more challenging would be to ask what effect the introduction of the equivalence relation has on this phase transition?

\vspace{\baselineskip}
\noindent{\bf Acknowledgements: }The authors are very grateful to Pierre Lescanne for fruitful discussions about this project, to Brigitte Chauvin and Dani{\`e}le Gardy for proof-reading an early version of this manuscript and to the anonymous referees for their very insightful comments. The second author also wishes to thank EPSRC for support through the grant EP/K016075/1.

\bibliographystyle{spmpsci}
\bibliography{boolean}

\end{document}